\documentclass[11pt]{article}

\usepackage{amsmath,amssymb,amsthm,hyperref,color}
\usepackage{enumerate}
\usepackage{verbatim}
\usepackage{bookmark}

\usepackage{thmtools}
\usepackage{thm-restate}
\usepackage[noadjust]{cite}

\usepackage[T1]{fontenc}
\usepackage{libertine} 
\usepackage[libertine]{newtxmath}

\hypersetup{colorlinks=true,citecolor=blue, linkcolor=blue, urlcolor=blue}{\tiny }
\usepackage[margin = 1in]{geometry}

\newtheorem{lemma}{Lemma}[section]
\newtheorem{theorem}[lemma]{Theorem}

\newtheorem{proposition}[lemma]{Proposition}
\newtheorem{claim}[lemma]{Claim}

\newtheorem{corollary}[lemma]{Corollary}

\theoremstyle{remark}
\newtheorem{remark}{Remark}
\theoremstyle{definition}
\newtheorem{definition}[lemma]{Definition}

\DeclareMathOperator*{\Pp}{{\mathrm P}}

\usepackage{bm}
\usepackage{bbm}
\DeclareMathOperator*{\1}{\mathbbm{1}}

\def\E{\mathrm{E}}
\def\Ent{\mathrm{Ent}}
\def\var{\mathrm Var}

\def\gap{\textsc{GAP}}
\def\alphamin{\alpha_{min}}

\def\potts{\mu_{\mathrm{\footnotesize{Potts}}}}
\def\ising{\mu_{\mathrm{\footnotesize{Ising}}}}

\def\Tcoup{T_{\mathrm coup}}

\renewcommand{\epsilon}{\varepsilon}

\newcommand{\R}{\mathbb{R}}

\date{\today}

\begin{document}

\author{
	Antonio Blanca
	\thanks{Pennsylvania State University.
		Email: ablanca@cse.psu.edu.
		Research supported in part by NSF grant CCF-2143762.}
	\and\
	Xusheng Zhang\thanks{Pennsylvania State University.
		Email: xushengz@psu.edu.
		Research supported in part by NSF grant CCF-2143762.}
}
\title{Rapid mixing of global Markov chains via spectral independence: the unbounded degree case}

\maketitle

\begin{abstract}

We consider spin systems on general $n$-vertex graphs of unbounded degree and explore the effects of spectral independence on the rate of convergence to equilibrium of {global} Markov chains. Spectral independence is a novel way of quantifying the decay of correlations in spin system models, which has significantly advanced the study of Markov chains for spin systems. 
We prove that whenever spectral independence holds, the popular Swendsen--Wang dynamics for the $q$-state ferromagnetic Potts model
on graphs of maximum degree $\Delta$, where $\Delta$ is allowed to grow with $n$, 
converges in $O((\Delta \log n)^c)$ steps where $c > 0$ is a constant independent of $\Delta$ and $n$.
We also show a similar mixing time bound for the block dynamics of general spin systems, again assuming that spectral independence holds.
Finally, for \emph{monotone} spin systems such as the Ising model and the hardcore model on bipartite graphs, we show that spectral independence implies that the mixing time of the systematic scan dynamics is $O(\Delta^c \log n)$ for a constant $c>0$ independent of $\Delta$ and $n$.
Systematic scan dynamics are widely popular but are notoriously difficult to analyze. 
Our result implies optimal $O(\log n)$ mixing time bounds for any systematic scan dynamics of the ferromagnetic Ising model on general graphs up to the tree uniqueness threshold.
Our main technical contribution is an improved factorization of the entropy functional: this is the common starting point for all our proofs.
Specifically, we establish the so-called $k$-partite factorization of entropy 
with a constant that depends polynomially on the maximum degree of the graph.

\end{abstract}

\thispagestyle{empty}

\newpage

\setcounter{page}{1}

\section{Introduction}

Spectral independence is a powerful new approach for quantifying the decay of correlations in spin system models. Initially introduced in \cite{ALO20}, this condition has revolutionized the study of Markov chains
for spin systems. In a series of important and recent contributions, spectral independence has been shown to be instrumental in determining the convergence rate of the Glauber dynamics, the simple single-site update Markov chain that updates the spin at a randomly chosen vertex in each step.

The first efforts in this series (see \cite{ALO20,CLV20, CLV21}) showed that spectral independence implies optimal $O(n \log n)$ mixing of the Glauber dynamics on $n$-vertex graphs of bounded degree for general spin systems. The unbounded degree case was studied in~\cite{CFYZ22a,CFYZ22b,AJKPV22,JPV21}, while~\cite{BCCPSV22} explored the effects of this condition on the speed of convergence of global Markov chains (i.e., Markov chains that update the spins of a large number of vertices in each step) in the bounded degree setting.
Research exploring the applications of spectral independence is ongoing.
We contribute to this line of work by investigating how spectral independence affects the speed of convergence of \emph{global Markov chains} for general spin systems on graphs of unbounded degree.

A \emph{spin system} is defined on a graph $G=(V,E)$.
There is a set~$\mathcal S = \{1,\dots,q\}$ of spins or colors, and configurations are assignments of spin values from $\mathcal S$
to each vertex of $G$.
The probability of a configuration~$\sigma \in \mathcal S^V$ is given by the  Gibbs
distribution:
\begin{equation}\label{eqn:gibbs}
\mu(\sigma) =\frac{e^{-H(\sigma)}}{Z},
\end{equation}
where the normalizing factor $Z$ is known as the partition function, and the Hamiltonian~$H: \mathcal S^V \rightarrow \R$
contains terms that depend on the spin values at each vertex 
(a ``vertex potential'' or ``external field'') and
at each pair of adjacent vertices (an ``edge potential''); 
see Definition \ref{def:spin-system}.
A widely studied spin system, and one that we will pay close attention to in this paper, is the ferromagnetic Potts model, where for a real parameter $\beta > 0$, associated with inverse temperature in physical applications, the Hamiltonian is given~by:
\[
H(\sigma) = -\beta \sum_{\{u,v\} \in E}  \1(\sigma_u = \sigma_v).
\]
The classical ferromagnetic Ising model corresponds to the $q=2$ case. 
(In this variant of the Potts model,
the Hamiltonian only includes edge potentials, and 
there is no external field.)
We shall use $\ising$ and $\potts$ for the Gibbs distributions corresponding to the Ising and Potts models. Other well-known, well-studied spin systems include uniform proper colorings and the hardcore model.

Spin systems provide a robust framework for studying interacting systems of simple elements and have a wide range of applications in computer science, statistical physics, and other fields. In such applications, generating samples from the Gibbs distribution \eqref{eqn:gibbs} is a fundamental computational task and one in which Markov chain-based algorithms have been quite successful. A long line of work dating back to the 1980s relates the speed of convergence of Markov chains to various forms of decay of correlations in the model. Spectral independence, defined next, captures the decay of correlations in a novel way.

Roughly speaking, spectral independence holds when the spectral norm of a ``pairwise'' influence matrix is bounded.
To formally define it, let us begin by introducing some notations.
Let $\Omega \subseteq \mathcal S^V$ be the support of $\mu$: the set of configurations $\sigma$ such that $\mu(\sigma) > 0$.
A \emph{pinning} $\tau$ on a subset of vertices $\Lambda \subseteq V$ is a fixed partial configuration on $\Lambda$;
i.e., a spin assignment from $\mathcal S^\Lambda$ to the vertices of $\Lambda$. 
For a pinning $\tau$ on $\Lambda \subseteq V$ and $U \subseteq V\setminus\Lambda$, we let 
$\Omega_{U}^\tau = \{\sigma_U \in \mathcal S^U: \mu(\sigma_U \mid \sigma_\Lambda = \tau ) > 0\}$ 
be the set of partial configurations on $U$ that are consistent with the pinning $\tau$.
We write $\Omega^\tau_u = \Omega^\tau_{\{u\}}$ if $u$ is a single vertex.
Let
\[
\mathcal P^\tau := \{(u, s): u\notin \Lambda, s\in \Omega^\tau_u \}
\] 
denote the set of consistent vertex-spin pairs in $\Omega^\tau_{V\setminus \Lambda}$ under $\mu$.
For each $\Lambda \subseteq V$ and pinning $\tau$ on $\Lambda$, 
we define the \emph{signed pairwise influence matrix} $\Psi^\tau_\mu \in \mathbb{R}^{\mathcal P^\tau \times \mathcal P^\tau}$ 
to be the matrix with entries: 
\[
\Psi^\tau_\mu((u,a), (v,b)) = \mu(\sigma_v = b \mid \sigma_u = a, \sigma_\Lambda = \tau) 
- \mu(\sigma_v = b \mid  \sigma_\Lambda = \tau)
\]  
for $u \neq v$, and $\Psi^\tau_\mu((u,a), (u,b)) = 0$ otherwise.

\begin{definition}[Spectral Independence]
\label{def:si}
A distribution $\mu$ satisfies \emph{$\eta$-spectral independence} if 
for every subset of vertices $\Lambda \subseteq V$ and every pinning $\tau\in \Omega_\Lambda$, 
the largest eigenvalue of
the signed pairwise influence matrix $\Psi^\tau_\mu$, denoted $\lambda_1(\Psi^\tau_\mu)$, satisfies
$
\lambda_1(\Psi^\tau_\mu) \le \eta.
$
\end{definition}
\noindent
There are several definitions of spectral independence in the literature;
we use here the one from \cite{CGSV21}. 

We show that spectral independence implies new upper bounds on the mixing time of several well-studied {global} Markov chains in the case where the maximum degree $\Delta$ of the underlying graph $G=(V,E)$ is unbounded; i.e., $\Delta \rightarrow \infty$ with $n$. 
The mixing time is defined as the number of steps required for a Markov chain to reach a distribution close in total variation distance to 
its stationary distribution, assuming a worst possible starting state; a formal definition is given in Section~\ref{sec:tmix}.
The global Markov chains we consider include the Swendsen--Wang dynamics for the ferromagnetic $q$-state Potts, the systematic scan dynamics for monotone spin systems, and the block dynamics for general spin systems. These three dynamics are among the most popular and well-studied global Markov chains and present certain advantages (e.g., faster convergence and amenability to parallelization) to the Glauber dynamics.

\subsection{The Swendsen--Wang dynamics}
A canonical example of a global Markov chain is the Swendsen--Wang (SW) dynamics for the ferromagnetic $q$-state Potts model.
The SW dynamics transitions from a configuration $\sigma_t$ to $\sigma_{t+1}$ by:
\begin{enumerate}
    \item For each edge $e = \{u,v\}\in E$, if $\sigma_t(u) = \sigma_t(v)$, independently include $e$ in the set $A_t$ with probability~$p = 1 - e^{-\beta}$;
    \item Then, independently for each connected component $\mathcal C$ in $(V,A_t)$, draw a spin $s \in \{1, \dots ,q\}$ uniformly at random and set $\sigma_{t+1}(v)= s$ for all $v\in \mathcal C$. 
\end{enumerate}
The SW dynamics is ergodic and reversible with respect to $\potts$ and thus converges to it.  
This Markov chain originated in the late 1980s~\cite{SW} as an alternative to the Glauber dynamics, which mixes exponentially slowly at low temperatures (large $\beta$).
The SW dynamics bypasses the key barriers that cause the slowdown of the Glauber dynamics at low temperatures.
For the Ising model ($q=2$), for instance, it was recently shown to converge in $\mathrm{poly}(n)$ steps on any $n$-vertex graph for any value of $\beta > 0$~\cite{GuoJer}. (The conjectured mixing time is $\Theta(n^{1/4})$, but we seem to be far from proving such a conjecture.)
For $q \ge 3$, on the other hand, the SW dynamics can converge exponentially slowly at certain  ``intermediate'' temperatures regimes corresponding to first-order phase transitions; see~\cite{GoJe,BCT,GL1,GLP,PottsRGMetastabilityCMP}.

Recently, $\eta$-spectral independence (with $\eta = O(1)$) was shown to imply that the mixing time of the SW dynamics is $O(\log n)$ on graphs of maximum degree $\Delta = O(1)$, i.e.,  bounded degree graphs~\cite{BCCPSV22}. This mixing time bound is optimal since the SW dynamics requires $\Omega(\log n)$ steps to mix in some cases where $\eta$ and $\Delta$ are both $O(1)$~\cite{BCPSV21,BCSV22}. However, it does not extend to the unbounded degree setting since the constant factor hidden by the big-$O$ notation depends exponentially on the maximum degree $\Delta$; this is the case even when $\eta = O(1)$ and $\beta \Delta = O(1)$.
Our first result provides a mixing time bound that depends only polynomially on $\Delta$. 

\begin{theorem}
	\label{thm:SW:intro}
    Let $q\ge 2$, $\beta > 0$, $\eta$ > 0 and $\Delta \ge 3$.
    Suppose $G= (V,E)$ is an $n$-vertex graph of maximum degree $\Delta$. 
	Let $\potts$ be the Gibbs distribution of the $q$-state ferromagnetic Potts model on $G$ with parameter $\beta$. 
	If  $\potts$ is $\eta$-spectrally independent with $\eta =O(1)$ and $\beta \Delta = O(1)$, then
    there exists a constant $c > 0$ such that
    the mixing time of the SW dynamics satisfies
	$T_{mix}(P_{SW}) = O\big((\Delta \log n)^{c} \big).$
\end{theorem}

\noindent
The constant $c$ has a near linear dependency on $\eta$ and $\beta \Delta$; a more precise statement of Theorem~\ref{thm:SW:intro} with a precise expression for $c$ is given in~Theorem~\ref{thm:SW}.

Despite the expectation that the SW dynamics mixes in $O(\log n)$ steps in weakly correlated systems (i.e., when $\beta \Delta$ is small), proving sub-linear upper bounds on its mixing time has been difficult.
Recently, 
various forms of decay of correlation (e.g., strong spatial mixing, entropy mixing, and spectral independence)
have been used to obtain $O(\log n)$ bounds 
for the mixing time of the SW dynamics
on cubes of the integer lattice graph $\mathbb{Z}^d$, regular trees, and general graphs of bounded degree (see~\cite{BCPSV21,BCSV22,BCCPSV22}). 
However, for graphs of large degree, i.e., with $\Delta \rightarrow \infty$ with $n$, the only sub-linear mixing time bounds known either hold for the very distinctive mean-field model, where $G$ is the complete graph~\cite{GSVmf,BSmf}, or hold for very small values of $\beta$; i.e., $\beta \lesssim 1/(3\Delta)$~\cite{huber2003bounding}. 
Our results provide new sub-linear mixing time bounds for graph families of sub-linear maximum degree, provided $\eta =O(1)$ and $\beta \Delta = O(1)$. These last two conditions go hand-in-hand: in all known cases where $\eta = O(1)$, we also have $\beta \Delta = O(1)$.

On graphs of degree at most $\Delta$, $\eta$-spectral independence is supposed to hold with $\eta = O(1)$ whenever $\beta < \beta_u(q,\Delta)$, where $\beta_u(q,\Delta)$ is the threshold for the uniqueness/non-uniqueness phase transition on $\Delta$-regular trees. 
This has been confirmed for the Ising model ($q=2$) but not for the Potts model. Specifically, for the ferromagnetic Ising model, we have $\beta_u(2,\Delta) = \ln \frac{\Delta}{\Delta-2}$, and when $\beta \le (1-\delta)\beta_u(2,\Delta)$ for some $\delta \in (0,1)$, $\ising$ is $\eta$-spectrally independent with $\eta = O(1/\delta)$; see \cite{CLV20, CLV21}.
In contrast, for the ferromagnetic Potts model with $q \ge 3$,
there is no closed-form expression for $\beta_u(q,\Delta)$ (it is defined as the threshold value where an equation starts to have a double root), and for graphs of unbounded degree $\eta$-spectral independence is only known to hold  when $\beta \le \frac{2(1-\delta)}{\Delta}$.
As a result, we obtain the following corollary of Theorem~\ref{thm:SW:intro}.

\begin{corollary}
	\label{cor:SW:potts:intro}
	Let $\delta\in(0,1)$, $\Delta \ge 3$. Suppose that either $q = 2$ and $\beta < (1-\delta)\beta_u(2,\Delta)$, or
    $q \ge 3$ and $\beta \le \frac{2(1-\delta)}{\Delta}$.
    Then,
    there exists a constant $c = c(\delta) > 0$ such that
    the mixing time of the SW dynamics 
    for the $q$-state ferromagnetic Potts model
    on any  $n$-vertex graph of maximum degree $\Delta$ satisfies 
    $
    T_{mix}(P_{SW}) = O\big((\Delta \log n)^{c} \big).
    $
\end{corollary}

\noindent
We mention that other conditions known to imply spectral independence (e.g., those in~\cite{BGP16}) are not well-suited for the unbounded degree setting since under those conditions, the best known bound for $\eta$ depends polynomially on $\Delta$. 
For another application of Theorem~\ref{thm:SW:intro}, see Section~\ref{subsec:rg} where we provide a bound on the mixing of the SW dynamics on random graphs.

We comment briefly on our proof approach for Theorem~\ref{thm:SW:intro}.
A mixing time bound for the SW dynamics can be deduced from the so-called \emph{edge-spin} factorization of the entropy functional introduced in \cite{BCPSV21}.
It was noted there that this factorization, in turn, follows 
from a different factorization of entropy known as
\emph{$k$-partite factorization}, or  KPF. Spectral independence is known to imply KPF but with a loss of a multiplicative constant that depends exponentially on the maximum degree of the graph.
Our proof of Theorem~\ref{thm:SW:intro} follows this existing framework,
but pays closer attention to establishing KPF with an optimized constant with a better dependence on the model parameters.
This is done through a multi-scale analysis of the entropy functional; in each scale, we apply spectral independence to achieve a tighter KPF condition. 
Our new results for KPF not only hold for the Potts model, 
but also for a general class of spin systems, and we use it to establish new mixing time bounds for the systematic scan and block dynamics.

\subsection{The systematic scan dynamics}
Our next contribution pertains the \emph{systematic scan dynamics}, 
which is a family of Markov chains closely related to the Glauber dynamics in the sense that updates occur at single vertices sequentially.
The key difference is that the vertex updates 
happen
according to a predetermined ordering $\phi$ of the vertices instead of at random vertices. 
These dynamics offer practical advantages since there is no need to randomly select vertices at each step, thereby reducing computation time. 
Throughout the paper, we will consider the \emph{heat-bath} vertex updates in which a new spin is assigned to a vertex
by sampling from the conditional distribution at the vertex given the spins of its neighbors; this will be the case for
both the Glauber and systematic scan dynamics.

There is a folklore belief that the mixing time of the systematic scan dynamics (properly scaled) is closely related to that of the Glauber dynamics. However, analyzing this type of dynamics has proven very challenging (see, e.g., \cite{DGJ08,Hayes06,DGJnorms,DGJcolorings,PW,guoScans,BCSV18}),
and the best general condition under which the systematic scan dynamics is known to be optimally mixing is a Dobrushin-type condition due to Dyer, Goldberg, and Jerrum~\cite{DGJnorms}. The new developments on Markov chain mixing stemming from spectral independence have not yet provided new results for this dynamics, even for the bounded degree case where much progress has already been made. We show that spectral independence implies optimal mixing of the systematic scan dynamics for \emph{monotone} spin systems with \emph{bounded marginals}; we define both of these notions next.

\begin{definition}[Monotone spin system]
    \label{def:monotone}
    In a monotone system, there is a linear ordering of the spins at each vertex which
    induces a partial order $\preceq_q$ over the state space. 
    A spin system is \emph{monotone} with respect to the
    partial order $\preceq_q$ if for every $\Lambda \subseteq V$ and every pair of pinnings 
    $\tau_1 \succeq_q \tau_2$ on $V \setminus \Lambda$, the conditional distribution $\mu(\cdot\mid \sigma_\Lambda = \tau_1)$ stochastically dominates $\mu(\cdot\mid \sigma_\Lambda = \tau_2)$.
\end{definition}
\noindent
Canonical examples of monotone spin systems include the ferromagnetic Ising model and the hardcore model on bipartite graphs. As in earlier work (see~\cite{CLV20,CLV21,BCCPSV22}), our bounds on the mixing time will depend on a lower bound on the marginal probability of any vertex-spin pair. This is formalized as follows.

\begin{definition}[Bounded marginals]
The distribution $\mu$ is said to be \emph{$b$-marginally bounded}
if for every $\Lambda \subseteq V$ and pinning $\tau \in \Omega_{\Lambda}$, 
and each $(v,s)\in\mathcal P^\tau$,
we have
$
	\mu(\sigma_v = s \mid \sigma_\Lambda = \tau) \ge b.
$
\end{definition}

Before stating our result for the systematic scan dynamics of $b$-marginally bounded monotone spin systems, we note that 
this Markov chain updates in a single step each vertex once in the order prescribed by $\phi$. 
Under a minimal assumption on the spin system (the same one required to ensure the ergodicity of the Glauber dynamics), 
the systematic scan dynamics is ergodic. Specifically, when the spin system is totally-connected (see Definition~\ref{def:tc}), the systematic scan dynamics is ergodic.
Moreover, 
the systematic scan dynamics is not necessarily reversible with respect to $\mu$, so, as in earlier works, we work
with the symmetrized version of the dynamics in which, in each step, the vertices are updated according to $\phi$ first, and subsequently in the reverse order of $\phi$. The resulting dynamics, which we denote by $P_\phi$, is reversible with respect to $\mu$. Our main result for the systematic scan dynamics is the following. 

\begin{restatable}[]{theorem}{ssintro}
	\label{thm:SIIsing:intro}
    Let $b > 0$, $\eta$ > 0, and  $\Delta \ge 3$.
	Suppose $G= (V,E)$ is an $n$-vertex graph of maximum degree $\Delta$. 
	Let $\mu$ be the distribution of a totally-connected monotone spin system on $G$. 
	If $\mu$ is $\eta$-spectrally independent and $b$-marginally bounded,
	then there exists a universal constant $C > 0$ such that for any ordering $\phi$ 
	\[
	T_{mix}(P_{\phi}) = \Delta^{9+ 4\lceil\frac{2\eta}{b}\rceil} \cdot \left(\frac{C (\eta+1)^5 }{b^{6}} \right)^{2+\lceil\frac{2\eta}{b}\rceil} \cdot O(\log n).
	\]
\end{restatable}

\noindent
The bound in this theorem is tight: for a particular ordering $\phi$, we prove an $\Omega(\log n)$ mixing time lower bound that applies to settings where $\Delta$, $b$ and $\eta$ are all $\Theta(1)$; see Lemma~\ref{lemma:lbss}.

We present next several interesting consequences of Theorem~\ref{thm:SIIsing:intro}. First,
we obtain the following corollary using the known results about spectral independence for the ferromagnetic Ising model.

\begin{corollary}
	\label{cor:tuIsing:intro}
	Let  $\delta \in (0,1) ,\Delta \ge 3$ and $0 < \beta < (1-\delta)\beta_u(2,\Delta)$.
    Suppose $G= (V,E)$ is an $n$-vertex graph of maximum degree $\Delta$. 
    For any ordering $\phi$ of the vertices of $G$, 
    the mixing time of $P_\phi$ for the Ising model on $G$ with parameter $\beta$ 
    satisfies $T_{mix}(P_\phi) = O(\log n)$.
\end{corollary}

\noindent
The constant hidden by the big-$O$ notation is an absolute constant that depends only on the constant $\delta$, even when $\Delta$ depends on $n$.
This result, compared to the earlier conditions in~\cite{DGJ08,Hayes06,DGJnorms}, extends the parameter regime where the $O(\log n)$ mixing time bound applies; in fact, the parameter regime in Corollary~\ref{cor:tuIsing:intro} is tight, as the systematic scan dynamics undergoes an exponential slowdown when $\beta > \beta_u(2,\Delta)$~\cite{PW}.
We also derive results for the hardcore model on bipartite graphs; see Section~\ref{sec:sscor}.

Our next application concerns the specific but relevant case 
where the underlying graph is an $n$-vertex cube of the integer lattice graph $\mathbb{Z}^d$.
In this context,
it was proved in~\cite{BCSV18} that all systematic scan dynamics converge in $O(\log n (\log \log n)^2)$ steps
whenever a well-known condition known as 
\emph{strong spatial mixing (SSM)} holds.
A pertinent open question is whether SSM implies spectral independence. In fact, 
spectral independence is often proved
by adapting earlier arguments for establishing SSM (see, e.g.,~\cite{ALO20,CLV20}).
Recently, it was proved in \cite{CLMM23} that SSM on trees implies spectral independence on large-girth graphs.
We show that for \emph{general} spin systems on $\mathbb{Z}^d$, SSM
implies $\eta$-spectral independence with $\eta = O(1)$.

\begin{restatable}[]{lemma}{ssmtosiintro}
	\label{lemma:SSMtoSI}
	For a spin system on a $d$-dimensional cube $V \subseteq \mathbb{Z}^d$, SSM implies $\eta$-spectral independence, where $\eta=O(1)$.
\end{restatable}
\noindent
The formal definition of SSM is given later in Section~\ref{sec:ss}. Lemma~\ref{lemma:SSMtoSI} does not assume monotonicity for the spin system and could be of independent interest.
An interesting consequence of this lemma, when combined with~Theorem~\ref{thm:SIIsing:intro} is the following.
\begin{corollary}
\label{cor:monogrid:intro}
Let $d\ge 2$ and $b > 0$. For a $b$-marginally bounded monotone totally-connected spin system on a $d$-dimensional cube $V\subseteq\mathbb{Z}^d$, SSM implies that  
the mixing time of any systematic scan $P_\phi$ is $O(\log n)$.
\end{corollary}
\noindent For the ferromagnetic Ising model on $\mathbb{Z}^2$,
SSM is known to hold for all $\beta < \beta_c(2) = \ln(1+\sqrt{2})$ (see~\cite{CP20,MOS94,Alx98,BDC12}), so by Corollary~\ref{cor:monogrid:intro} we deduce that when $\beta < \beta_c(2)$, the mixing time of any systematic scan $P_\phi$ on an $n$-vertex square box of $\mathbb{Z}^2$ is $O(\log n)$; note that $\beta_c(2)  > \beta_u(2,2d)$, 
the corresponding tree uniqueness threshold.

We comment briefly on the techniques used to establish our results for the systematic scan dynamics.
Our starting point is again the $k$-partite factorization of entropy (KPF). Our improved bounds for KPF imply that a global Markov chain that updates a random independent set of vertices in each step is rapidly mixing.
We then use the censoring technique from \cite{FillKahn13, BCV18} to relate the mixing time of this Markov chain to that of the systematic scan dynamics. 
To establish Lemma~\ref{lemma:SSMtoSI}, 
we use SSM to construct a contractive coupling for a particular 
Markov chain.
Our Markov chain is similar to the one from~\cite{DSVW04}, but modified to update rectangles instead of balls, and thus match the variant of SSM that holds up to the critical threshold for the Ising model on $\mathbb{Z}^2$.
This contractive coupling is then used to establish spectral independence using the machinery from \cite{BCCPSV22}.

\subsection{The block dynamics}
\label{sec:blocks:intro}

Our final result concerns a family of Markov chains known as the \emph{block dynamics}. They are a natural generalization of the Glauber dynamics where a random subset of vertices (instead of a random vertex) is updated in each step. More precisely, 
let $\mathcal{B} := \{B_1, \dots, B_K\}$ be a collection of subsets of vertices (called {blocks}) such that $V=\cup_{i=1}^K B_i$.
Let $\alpha$ be a distribution over $\mathcal{B}$.
The \emph{(heat-bath) block dynamics} with respect to $(\mathcal{B}, \alpha)$
is the Markov chain that, in each step,
given a spin configuration $\sigma_t$,
selects $B_i \in \mathcal{B}$ 
according to the distribution $\alpha$ and updates the configuration on $B_i$
with a sample from the $\mu(\cdot \mid \sigma_t(V\setminus B_i))$; that is, from the conditional distribution 
on $B_i$ given the spins of $\sigma_t$ in $V\setminus B_i$.
We denote this Markov chain (and its transition matrix) by $P_{\mathcal{B},\alpha}$.
When the $B_i$'s are each single vertices, and $\alpha$ is a uniform distribution over the blocks in $\mathcal{B}$, we obtain the Glauber dynamics.
Our result for the mixing time of the block dynamics is the following.

\begin{theorem}
	\label{thm:blocks:intro}
    Let $b>0$, $\eta>0$ and $\Delta \ge 3$.
    Suppose $G= (V,E)$ is an $n$-vertex graph of maximum degree $\Delta$. 
Let $\mu$ be a Gibbs distribution of a totally-connected spin system on $G$.    
    Let $\mathcal{B} := \{B_1, \dots, B_K\}$ be any collection of blocks such that $V=\cup_{i=1}^K B_i$, and let $\alpha$ be a distribution over $\mathcal{B}$.
    If $\mu$ is $\eta$-spectrally independent and $b$-marginally bounded, 
    then there exists a universal constant $C > 0$ such that
    the mixing time of block dynamics $P_{\mathcal{B},\alpha}$
    satisfies:
	\[T_{mix}(P_{\mathcal{B},\alpha}) = O\Big( \alphamin^{-1} \cdot \Big(\frac{C  (\eta+1)^5  \Delta \log n }{b^6}\Big)^{3+\lceil\frac{2\eta}{b}\rceil}\Big),\]
    where $\alphamin= \min_{v\in V} \sum_{B\in\mathcal{B}} \alpha_B$.
\end{theorem}
\noindent
See Theorem~\ref{thm:GBF} for a more precise statement.
Previous results for the block dynamics only apply to the bounded degree case~\cite{BCSV22,CP20,BCCPSV22}, so Theorem~\ref{thm:blocks:intro} provides the first bounds 
for its mixing time in the unbounded degree setting.

\section{Preliminaries}

This section provides several definitions and background results we will refer to in our proofs.

\subsection{Mixing times and modified log-Sobolev inequalities}
\label{sec:tmix}
Let $P$ be an irreducible and aperiodic  (i.e., ergodic) Markov chain
with state space $\Omega$ and stationary distribution $\mu$. Let us assume that
$P$ is reversible with respect to $\mu$,
and let
\[ 
d(t):=\max_{x\in \Omega } \|P^t(x,\cdot) - \mu\|_{TV} := \max_{x\in \Omega } \max_{A\subseteq \Omega} |P^t(x,A)-\mu(A)|,
\] where $P^t(x,\cdot)$ denotes the distribution of the chain at time $t$ assuming $x \in \Omega$ as the starting state; $\|\cdot\|_{TV}$ denotes the total variation distance.
Note that with a slight abuse of notation we use $P$ for both the Markov chain and its transition matrix.
For $\epsilon>0$, 
let $$T_{mix}(P, \epsilon) := \min \{t>0: d(t)\le \epsilon \},$$ and
the \emph{mixing time of $P$} is defined as $T_{mix}(P) = T_{mix}(P, 1/4)$.

For functions $f,g:\Omega\rightarrow \mathbb{R}$, 
the \emph{Dirichlet form} of a reversible Markov chain $P$ with stationary distribution $\mu$ is defined as 
\[
\mathcal{E}_P(f,g) = \langle f, (I-P) g\rangle_{\mu} = \frac{1}{2} \sum_{x,y\in \Omega} 
\mu(x) P(x,y) (f(x)- f(y)) (g(x) - g(y)),
\]
where $\langle f, g\rangle_{\mu}:=\sum_{x\in \Omega} f(x)g(x) \mu(x)$.

The spectrum of the ergodic and reversible Markov chain $P$ is real, and we let $1 = \lambda_1 > \lambda_2 \ge \dots \ge \lambda_{|\Omega|} \ge -1$ denote its eigenvalues. The (absolute) spectral gap of $P$ is defined by $\gap(P) = 1 - \max\{|\lambda_2|,|\lambda_{|\Omega|}|\}$.
When $P$ is positive semidefinite, we have
\[
	\gap(P) =1 - \lambda_2 = \inf\left\{ \frac{\mathcal{E}_P(f,f)}{\langle f, f\rangle_{\mu}} \mid f:\Omega\rightarrow \mathbb{R}, \langle f, f\rangle_{\mu}\neq 0 \right\}.
\]
For $P$ reversible and ergodic, we have 
the following standard comparison between the spectral gap and the mixing time 
\begin{equation}
    \label{eq:gaptmix}
    T_{mix}(P, \epsilon) = \frac{1}{\gap(P)}\cdot \log \big(\frac{1}{\epsilon\mu_{min}}\big),
\end{equation}
where $\mu_{min} := \min_{x\in \Omega} \mu(x)$. 

The expected value of a function $f: \Omega \rightarrow \mathbb{R}_{\ge 0}$
with respect to $\mu$ is defined as $\E_\mu[f] = \sum_{x\in\Omega} f(x) \mu(x)$.
Similarly,
the entropy of the function 
with respect to $\mu$ is given by 
\[\Ent_\mu(f) := \E_\mu\Big[f \log \frac{f}{\E_\mu [f]}\Big]
= \E_\mu[f \log f] - \E_\mu[f \log(\E_\mu[f])].
\]  
We say that the Markov chain $P$ satisfies a \emph{modified log-Sobolev inequality} (MLSI) with constant $\rho$ if  
for every function $f:\Omega\rightarrow \mathbb{R}_{\ge 0}$, 
\[
\rho \cdot \Ent_\mu(f) \le  \mathcal{E}_P(f,\log f).
\]
The smallest $\rho$ satisfying the inequality above is 
called the \emph{modified log-Sobolev constant} of $P$ and is denoted by $\rho(P)$.
A well-known general relationship (see \cite{DS96, BT03}) shows that 
\begin{equation}
	\label{eq:gapmlsi}
	\frac{1-2\mu_{min}}{\log(1/\mu_{min} - 1)} \gap(P) 
	\le \rho(P) \le 2\gap(P).
\end{equation}
For distributions $\mu$ and $\nu$ over $\Omega$, 
the relative entropy of $\nu$ with respect to $\mu$, denoted as $\mathcal H(\nu \mid \mu)$, is defined as
$\mathcal H(\nu \mid \mu) := \sum_{x \in \Omega} \nu(x) \log \frac{\nu(x)}{\mu(x)}$.
A Markov chain $P$ with stationary distribution $\mu$ is said to satisfy discrete \emph{relative entropy decay} with rate $r > 0$ if for all distributions $\nu$:
\begin{equation}
\label{eq:dec}
    \mathcal H(\nu P \mid \mu) \le (1-r) \mathcal H(\nu \mid \mu).    
\end{equation}
It is a standard fact (see, e.g., Lemma 2.4 in~\cite{BCPSV21}) that when~\eqref{eq:dec} holds, then $\rho(P) \ge r$, and 
\begin{equation}
\label{eq:tmix}
    T_{mix}(P, \epsilon) \le \frac{1}{r} \cdot  \Big(\log \log \big(\frac{1}{\mu_{min}}\big) + \log \big(\frac{1}{2\epsilon}\big) \Big).
\end{equation}

\subsection{General spin system}

We provide next a general definition for spin systems and introduce the notion of totally-connected systems.

\begin{definition}[Spin system]
\label{def:spin-system}
Let $G=(V,E)$ be a graph and $\mathcal S= \{1,\dots, q\}$ a set of spins.  
Let $\Omega \subseteq \mathcal S^{V}$ be the set of possible spin configurations on $G$.
We write $\sigma_v$ for the spin assigned to $v$ by $\sigma$.
Given a configuration $\sigma \in \Omega$ and a subset $\Lambda$ of $V$, 
we write $\sigma_\Lambda \in \mathcal S^{\Lambda}$ 
for the configuration of $\sigma$ restricted to $\Lambda$. 
For a subset of vertices $\Lambda \subseteq V$, 
a \emph{boundary condition}  $\tau$ 
is an assignment of spins to (some) vertices in outer vertex boundary $\partial \Lambda \subseteq V \setminus \Lambda$ of $\Lambda$; namely, 
$\tau: (\partial\Lambda)_{\tau} \rightarrow \mathcal S$, with $(\partial\Lambda)_{\tau} \subseteq \partial \Lambda$.
Note that a boundary condition is simply a pinning of a subset of vertices identified as being in the boundary of $G$.
Given a boundary condition $\tau:(\partial V)_{\tau} \rightarrow \mathcal S$, the \emph{Hamiltonian} $H:\Omega\rightarrow \mathbb{R}$ of a spin system is defined as
\begin{equation}
\label{eq:spinsystem}
H(\sigma) = -\sum_{\{v,u\}\in E} K(\sigma_v, \sigma_u) 
- \sum_{\{v,u\}\in E: u\in V, v\in (\partial V)_{\tau}} K(\sigma_v, \tau_v) 
-\sum_{v\in V} U(\sigma_v),
\end{equation}
where $K:\mathcal S\times\mathcal S\rightarrow \mathbb{R}$ and $U:\mathcal S\rightarrow \mathbb{R}$ are
respectively the \emph{symmetric edge interaction potential function} and the 
\emph{spin potential function} of the system.
The \emph{Gibbs distribution} of a spin system with Hamiltonian $H$ is defined as
\[
\mu(\sigma) = \frac{1}{Z_H} e^{-H(\sigma)},
\]
where $Z_H := \sum_{\sigma \in \Omega} e^{-H(\sigma)}$.
We use $\Omega$ for the set of configurations $\sigma$ satisfying
$\mu(\sigma) > 0$.
\end{definition}
The Potts model, as defined in the introduction, corresponds to the spin system with $q\ge 2$, $K(x,y)= \beta\cdot \mathbbm{1}(x = y)$, and $U(\sigma_v)= 0$ for all $v\in V$. 
We focus on the ferromagnetic Ising model where $\beta >0$ and $\mathcal{S} = \{-1, +1\}$. 
Another important spin system is the hardcore model that can be defined by setting $\mathcal{S} = \{1, 0\}$, 
$K(x,y)=\infty$ if $x=y=1$ and $K(x,y)=0$ otherwise, 
and $U(x)= \mathbbm{1}(x=1) \cdot \ln \lambda$, where $\lambda >0$ is referred to as the \emph{fugacity} parameter of the model.

We restrict attention to \emph{totally-connected} spin systems, as this ensures that the Glauber dynamics, the systematic scan dynamics, and the block dynamics are all irreducible Markov chains (and thus ergodic).

\begin{definition}
    \label{def:tc}
    For a subset $\mathcal C_U$ of partial configurations on $U \subseteq V$, 
    let $H[\mathcal C_U] = (\mathcal C_U, E[\mathcal C_U])$ be the induced subgraph where $E[\mathcal C_U]$ consists of all pairs of configurations 
    on $\mathcal C_U$ that differ at exactly one vertex.
    We say that $\mathcal C_U$ is connected when $H[\mathcal C_U]$ is connected.
For a pinning $\tau$ on $\Lambda \subseteq V$, 
we say $\Omega^\tau_{V \setminus \Lambda}$ is {connected} if $H[\Omega^\tau_{V \setminus \Lambda}]$ is connected.  
A distribution $\mu$ over $\mathcal{S}^V$ is \emph{totally-connected} if 
for every $\Lambda \subseteq V$ and every pinning $\tau$ on $\Lambda$, $\Omega^\tau_{V \setminus\Lambda}$ is connected.
\end{definition}

\section{Swendsen-Wang dynamics on general graphs}
\label{sec:globalising}

In this section, we consider the SW dynamics
for the $q$-state ferromagnetic Potts models on general graphs. In particular, we establish Theorem~\ref{thm:SW:intro} from the introduction, which is a direct corollary of the following more general result.

\begin{theorem}
	\label{thm:SW}
    Let $q\ge 2$, $\beta > 0$, $\eta$ > 0, $b > 0$, $\Delta \ge 3$, and $\chi \ge 2$.   
    Suppose $G= (V,E)$ is an $n$-vertex graph of maximum degree $\Delta$ and chromatic number $\chi$. 
	Let $\potts$ be the Gibbs distribution of the $q$-state ferromagnetic Potts model on $G$ with parameter $\beta$. 
	 If $\potts$ is $\eta$-spectrally independent and 
    $b$-marginally bounded,
    then there exists a universal constant $C > 1$ such that
    the modified log-Sobolev constant of the SW dynamics satisfies:
	$$\rho(P_{SW}) = \Omega\left(\frac{b^{2+6\kappa}}{\chi \cdot (C \Delta \log n )^{\kappa} \cdot (\eta+1)^{5\kappa}}\right),
	$$where $\kappa = 2 +\lceil\frac{2\eta}{b}\rceil $, and 	
 \[T_{mix}(P_{SW}) = O\big(\chi \cdot (C \Delta\log n)^{\kappa} \cdot  (\eta+1)^{5\kappa} b^{-2-6\kappa}  \cdot \log n \big).\]
\end{theorem}
\noindent
Theorem~\ref{thm:SW:intro} follows from this theorem by noting that 
$\chi \le \Delta$ and that
under the assumptions $\eta = O(1)$ and $\beta \Delta = O(1)$, we have $b = O(1)$ and $\kappa = O(1)$.
\begin{remark}
    When $\Delta$ is small, i.e., $ \Delta = o(\log n)$, we can obtain slightly better bounds on  
    $\rho(P_{SW})$ and $T_{mix}(P_{SW})$ and replace the
    $(C \Delta\log n)^{\kappa}$ factor
    by a factor of $(C\Delta)^{8 + 4\lceil\frac{2\eta}{b}\rceil}$.
\end{remark}
Before proving Theorem~\ref{thm:SW}, we provide a number of definitions and required background results in Section~\ref{sec:factorization}. 
We then give the proof of Theorem~\ref{thm:SW} in Sections~\ref{subsec:sw:main:thm},~\ref{sec:maint}, and~\ref{subsect:sw:aux}, and include some applications of this result in Section~\ref{subsec:sw:app}.

\subsection{Factorization of entropy}
\label{sec:factorization}

 We present next several factorizations of the entropy functional $\Ent_\mu(f)$, which are instrumental in establishing the decay of the relative entropy for the SW dynamics. 
We introduce some useful notations first.
For a pinning $\tau$ in $V\setminus\Lambda$ (i.e., $\tau\in \Omega_{V\setminus\Lambda}$), we let $\mu_\Lambda^\tau(\cdot) := \mu(\cdot\mid \sigma_{V\setminus \Lambda} = \tau)$.
Given a function $f:\Omega \rightarrow \mathbb{R}_{\ge0}$, 
subsets of vertices $B \subseteq \Lambda \subset V$, 
and $\tau \in \Omega_{V \setminus \Lambda}$, 
	the function $f^{\tau}_B  : \Omega_{B}^\tau \rightarrow\mathbb{R}_{\ge0}$ is defined by: 
	$$
	f^{\tau}_B (\sigma) = \E_{\xi \sim \mu^\tau_{\Lambda \setminus B}} [f(\tau \cup \xi \cup \sigma)].
	$$If $B= \Lambda$, we often write $f^\tau$ for $f^{\tau}_B $, and if $\tau = \emptyset$, 
then we use $f_B$ for $f^{\tau}_B $. 
We use $\Ent^\tau_B(f^\tau)$ to denote $\Ent_{\mu^\tau_B}(f^\tau)$, and
if the pinning $\tau$ on $V\setminus B$ is from a distribution $\pi$ over $\Omega_{V\setminus B}$,
we use $\E_{\tau\sim\pi}[\Ent^\tau_B(f^\tau)]$ to denote the expected value
of the function $f$ on $S$ over the random pinning $\tau$.

Various forms of entropy factorization arise from bounding $\Ent_\mu (f)$ by different (weighted) sums of restricted
entropies of the function $f$.
The first one we introduced, is the so-called \emph{$\ell$-uniform block factorization of entropy} of \emph{$\ell$-UBF}.
For an integer $\ell \le n$, $\ell$-UBF holds for 
$\mu$ with constant $C_{\mathrm{UBF}}$ if for all functions $f : \Omega \rightarrow \mathbb{R}_{\ge0}$, 
\begin{equation}
\label{eq:ubf}
\frac{\ell}{n} \cdot \Ent_\mu (f) \le C_{\mathrm{UBF}} \cdot \frac{1}{\binom{n}{\ell}} \sum_{S \in \binom{V}{\ell}}
\E_{\tau \sim \mu_{V \setminus S}} \left[\Ent^\tau_{S} (f^\tau) \right],
\end{equation}
where $\binom{V}{\ell}$ denotes the collection of all subsets of $V$ of size $\ell$.
An important special case is when $\ell=1$, in which case~\eqref{eq:ubf} is called \emph{approximate tensorization of entropy (AT)}; this special case has been quite useful for establishing optimal mixing time bounds for the Glauber dynamics in various settings (see, e.g., \cite{Marton19,CMT15,Cesi,Martinelli}).
In recent works, a key step for obtaining AT has been to first establish $\ell$-UBF for some large $\ell$. The following result will be useful for us. 
\begin{theorem}[\cite{CLV21}, \cite{BCCPSV22}]
	\label{thm:SI2UBF}
 Let $b$ and $\eta$ be fixed.
 For ~$\theta\in(0,1)$ and $n \ge \frac{2}{\theta} (\frac{4\eta}{b^2} + 1)$, 
 the following holds.
	If the Gibbs distribution $\mu$ of a totally-connected spin system
	on an $n$-vertex graph 
	is $\eta$-spectrally independent and $b$-marginally bounded, then 
	$\lceil \theta n \rceil$-UBF holds with 
    $C_{\mathrm{UBF}}=(e/\theta)^{\lceil\frac{2\eta}{b}\rceil}$. 
\end{theorem}
\noindent

Another useful notion is the \emph{$k$-partite factorization of entropy} or KPF. 
Let $ U_1, \dots, U_k$ be $k$ disjoint independent sets of $V$ such that
$\bigcup_{i=1}^k U_i = V$. We say $\mu$ satisfies KPF with constant $C_{\mathrm{KPF}}$ if for all functions 
$f : \Omega \rightarrow \mathbb{R}_{\ge0}$, 
$$
\Ent_\mu (f) \le C_{\mathrm{KPF}} \sum_{i=1}^k
\E_{\tau \sim \mu_{V \setminus U_i}} \left[\Ent^\tau_{U_i} (f^\tau) \right].
$$KPF was introduced in~\cite{BCCPSV22}, where it was used to  analyze global Markov chains.
The interplay between KPF and UBF is intriguing and is further explored in this paper.

\subsection{Proof of main result for the SW dynamics: Theorem~\ref{thm:SW}}
\label{subsec:sw:main:thm}

The main technical contribution in the proof of
Theorem~\ref{thm:SW} is establishing KPF with a better (i.e., smaller) constant $C_{\mathrm{KPF}}$.
As in~\cite{BCCPSV22}, KPF is then used to derive an improved ``edge-spin'' factorization 
of entropy which is known to imply the desired bounds on the modified log-Sobolev constant and on
the mixing time of the SW dynamics.

\begin{theorem}
	\label{thm:SI2KPF}
	For a totally-connected and $b$-marginally bounded Gibbs distribution $\mu$ that satisfies $\eta$-spectral independence 
	on an $n$-vertex graph $G=(V,E)$ of maximum degree $\Delta \ge 3$, 
    if ~$b$ and $\eta$ are constants independent of $\Delta$ and $n$, 
    then there exists a constant $c=c(\eta,b)>0$ such that
    $k$-partite factorization of entropy holds for $\mu$ with constant
    $C_{\mathrm{KPF}} =(\Delta \log n)^{c}$.
	Specifically, 
 for a set of $k$ disjoint independent sets $V_1, \dots, V_k$ such that $\bigcup_{j=1}^k V_j = V$, we have
  \begin{align}
	\Ent_\mu(f) &\le 
          \Big(\frac{C (\eta+1)^5 \Delta\log n}{b^6} \Big)^{\kappa} 
	\cdot \sum_{j=1}^{k} \E_{\tau\sim \mu_{V \setminus V_j}}[\Ent^\tau_{V_j} (f^\tau)],~\text{and} \label{eq:SI2KPF2}\\
     \Ent_\mu(f) &\le  \Big(\frac{C (\eta+1)^5 \Delta^{4}}{b^{6}} \Big)^{\kappa} \cdot  
	 \sum_{j=1}^{k} \E_{\tau\sim \mu_{V \setminus V_j}}[\Ent^\tau_{V_j} (f^\tau)], \label{eq:SI2KPF1}
\end{align}
where $\kappa := 2+\lceil\frac{2\eta}{b}\rceil$ and $C > 0$ is a universal constant. 
\end{theorem}
\begin{remark}
	\label{remark:kpf}
 Let  $\mathcal{B} = \{B_1,\dots, B_k\}$ be a collection of disjoint independent sets such that $V = \bigcup_{i=1}^k B_i$.
The independent set dynamics $P_{\mathcal{B}}$ 
is a heat-bath block dynamics w.r.t. $\mathcal{B}$ and a uniform distribution over $\mathcal{B}$.	If $\mu$ satisfies $k$-partite factorization of entropy with $C_{\mathrm{KPF}}$,
    then $P_\mathcal{B}$ satisfies a 
    relative entropy decay with rate $r\ge 1/(k\cdot C_{\mathrm{KPF}})$. 
    See Lemma~\ref{prop:GBF} for the more general statement.
\end{remark}

As mentioned, KPF was first studied in \cite{BCCPSV22}; 
the constant proved there was  
\[C_{\mathrm{KPF}} = b^{-O(\Delta)} \cdot (\Delta/b)^{O(\eta/b)},\] 
so our new bound improves the dependence on $\Delta$ from exponential to polynomial.
The proof of Theorem~\ref{thm:SI2KPF} is given in two parts. In Section~\ref{sec:maint}, we prove~\eqref{eq:SI2KPF2}, whereas~\eqref{eq:SI2KPF1} is proved in Appendix~\ref{sec:SI2KPF1}.

With KPF on hand, the next step in the proof of Theorem~\ref{thm:SW} relies on the 
so-called edge-spin factorization of entropy.
Let $\Omega_J := \Omega \times \{0,1\}^{E}$ be the set of joint configurations $(\sigma, A)$ corresponding to pairs of a spin configuration $\sigma\in \Omega$ and an \emph{edge configuration} (a subset of edges in a graph) $A\subseteq E$.
For a $q$-state Potts model $\potts$ with parameter $p=1-e^{-\beta}$, 
we use $\nu$ to denote the \emph{Edwards-Sokal} measure on $\Omega_J$ given by
$$
\nu(\sigma,A) := \frac{1}{Z_J} (1-p)^{|E|-|A|} p^{|A|} \mathbf{1}(\sigma\sim A),
$$where $\sigma \sim A$ is the event that every edge in $A$ has its two endpoints with the same spin in $\sigma$, 
and $Z_J:=\sum_{(A,\sigma)\in \Omega_J}(1-p)^{|E|-|A|} p^{|A|} \mathbf{1}(\sigma\sim A)$ is a normalizing constant.
Let $\nu(\cdot \mid \sigma)$ and $\nu(\cdot \mid A)$ denote 
the conditional measures obtained from $\nu$ by fixing the spin configuration to be $\sigma$ or fixing the edge configuration to be $A$ respectively.
For a function $f : \Omega_J \rightarrow \mathbb{R}_{\ge0}$,
let $f^{\sigma} : \{0,1\}^{|E|} \rightarrow \mathbb{R}_{\ge0} $ be the function given by $f^\sigma(A)= f(\sigma \cup A)$, and
let $f^{A} : \Omega \rightarrow \mathbb{R}_{\ge0} $ be the function given by $f^A(\sigma)= f(\sigma \cup A)$.
We say that \emph{edge-spin factorization of entropy} holds with constant $C_{\mathrm{ES}}$ if for all functions $f:\Omega_J \rightarrow \mathbb{R}_{\ge0}$, 
\begin{equation}
	\label{eq:es}
	\Ent_\nu(f) \le C_{\mathrm{ES}} \left(
		\E_{(\sigma,A) \sim \nu} \left[
			\Ent_{A\sim \nu(\cdot \mid \sigma)} (f^\sigma) 
		 \right] + 
		 \E_{(\sigma,A) \sim \nu} \left[
			\Ent_{\sigma\sim \nu(\cdot \mid A)} (f^A) 
		 \right]
	 \right).
\end{equation}
\noindent 
The following result from \cite{BCCPSV22} will be useful for us.

\begin{lemma}[Theorem 6.1 \cite{BCCPSV22}]
	\label{lemma:kpf2es}
	Suppose the $q$-state ferromagnetic Potts model with parameter $\beta$ on a graph $G$ of maximum degree is $\Delta \ge 3$ satisfies KPF with constant $C_{\mathrm{KPF}}$.
	Then, the edge-spin factorization of entropy holds with constant $C_{\mathrm{ES}} = O(\beta \Delta k e^{\beta \Delta}) \cdot C_{\mathrm{KPF}} $. 
\end{lemma}
\begin{remark}
	The original bound for $C_{\mathrm{ES}}$ stated in \cite{BCCPSV22} is actually $O(\beta \Delta^2 e^{\beta \Delta}) \cdot C_{\mathrm{KPF}}$, 
	but in the proof there, one factor $k$ is replaced with $\Delta$ as its upper bound.
	Since we do not assume $\Delta$ to be a constant, we avoid such an upper bound.
	We also remark that the exponential dependence of $C_{\mathrm{ES}}$ on $\beta\Delta$ can probably be  improved, 
	but in our applications $\beta\Delta = O(1)$, so this would not represent a tangible improvement.
\end{remark}

The final ingredient in the proof of Theorem~\ref{thm:SW} is the following.

\begin{lemma}[Lemma 1.8 \cite{BCPSV21}]
	\label{lemma:es2sw}
	Suppose edge-spin factorization of entropy holds with constant $C_{\mathrm{ES}}$. 
	Then, the SW dynamics $P_{SW}$ satisfies the relative entropy decay with rate  
	$\Omega\left(\frac{1}{C_{\mathrm{ES}}}\right)$.
\end{lemma}
We are now ready to prove~Theorem~\ref{thm:SW}.

\begin{proof} [Proof of Theorem~\ref{thm:SW}]
	By Theorem~\ref{thm:SI2KPF},
    $\potts$ 
	satisfies
	$\chi$-partite factorization of entropy with constant
 \[
 C_{\mathrm{KPF}}=     \Big(\frac{C (\eta+1)^5 \Delta\log n}{b^6} \Big)^{\kappa},
 \]
where $C>0$ is a universal constant.
It follows from Lemma~\ref{lemma:kpf2es} and Lemma~\ref{lemma:es2sw} that 
the SW dynamics satisfies \eqref{eq:dec} with
\[
   r= \Omega\left(\frac{b^{6\kappa}}{\chi\beta \Delta  e^{\beta\Delta} \cdot C^\kappa (\eta+1)^{5\kappa} \cdot (\Delta\log n)^{\kappa}  }\right).
\] 
Note that $b\le q^{-1}e^{-\beta\Delta}$,
and so $\beta \Delta  e^{\beta\Delta}\le e^{2\beta \Delta} \le b^{-2}$.
Therefore, we obtain the desired bound for MLSI constant, and the mixing time bound follows from \eqref{eq:tmix}.
\end{proof}

\subsection{Proof of the main technical theorem:~Theorem~\ref{thm:SI2KPF}}
\label{sec:maint}

Recall that given a function $f:\Omega \rightarrow \mathbb{R}_{\ge0}$, 
subsets of vertices $B \subseteq \Lambda \subset V$, 
and $\tau \in \Omega_{V \setminus \Lambda}$, 
	the function $f^{\tau}_B : \Omega^\tau_{B} \rightarrow\mathbb{R}_{\ge0}$ is defined by 
	$$
	f^{\tau}_B (\sigma) = \E_{\xi \sim \mu^\tau_{\Lambda \setminus B}} [f(\tau \cup \xi \cup \sigma)].
	$$
In the proof of Theorem~\ref{thm:SI2KPF} we use several facts, which we compile next.

Let $S\subseteq V$ be a subset of vertices.
Let $S_1, \dots, S_m$ $\subseteq V$ denote the connected components of $S$.
For a vertex $v\in V$,
let $C_S(v)$ the unique connected component $S_i$ that contains $v$, if such component exists, otherwise set $C_S(v)$ to be the empty set. When $S$ is chosen uniformly at random
among all subsets of size $\lceil\theta n\rceil$, the following exponential tail bound 
for $|C_S(v)|$ was established in \cite{CLV21}.
\begin{lemma}[Lemma 4.3, \cite{CLV21}]
	\label{lemma:g1}
	Let $G=(V,E)$ be an $n$-vertex graph of maximum degree at most $\Delta$. Then for any $v\in V$ and every integer $k\ge 0$ we have 
	\begin{equation*}{\Pr}_S[|C_S(v)| = k]\le \frac{\ell}{n}\cdot (2e\Delta \theta)^{k-1},\end{equation*}
	where the probability ${\Pr}_S[\cdot]$ is taken over a uniformly random subset $S\subseteq V$ of size $\ell = \lceil\theta n \rceil$.
\end{lemma}
\begin{lemma}
\label{lemma:CGSVgap}
    Let $\mu$ be a totally-connected and $b$-marginally bounded distribution over $[q]^n$. If $\mu$ is $\eta$-spectrally independent, then the Glauber dynamics for $\mu$ has spectral gap at least
    \begin{equation}
        \left(\frac{2b^4}{(\lceil2\eta\rceil +2)^4} \cdot \frac{1}{n} \right)^{1+\lceil2\eta\rceil}.
    \end{equation}
\end{lemma}
\begin{remark}
    Lemma~\ref{lemma:CGSVgap} is similar to Theorem 1.3 in \cite{ALO20} (for 2-spin systems), Theorem 6 in \cite{CGSV21} (for colorings), and Theorem 3.2 in \cite{FGYZ21} (for a different notion of spectral independence).  For completeness, we provide a proof in Appendix~\ref{sec:addproofs}. 
\end{remark}
\begin{lemma}
        \label{lemma:gap2KPF}
     Let $\mu$ be a $b$-marginally bounded distribution over $[q]^n$. If the Glauber dynamics for $\mu$ has spectral gap $\gamma$, then $\mu$ satisfies KPF with constant
    \begin{equation}
      C_{KPF} \le \frac{3n\log (b^{-1})}{\gamma}.
    \end{equation}
\end{lemma}
The proof of Lemma~\ref{lemma:gap2KPF} is standard and is provided in Appendix~\ref{sec:addproofs}.
We proceed to prove \eqref{eq:SI2KPF2} from Theorem~\ref{thm:SI2KPF}.
With a slightly different argument, we will establish \eqref{eq:SI2KPF1} in Appendix~\ref{sec:SI2KPF1},
which is a better upper bound only when $\Delta =o( \log n)$. 

\begin{proof} [Proof of \eqref{eq:SI2KPF2} in Theorem~\ref{thm:SI2KPF}]
It follows from Lemma~\ref{lemma:CGSVgap} and Lemma~\ref{lemma:gap2KPF} that 
\[
C_{KPF} \le \frac{3(\lceil 2\eta \rceil + 2)^{4(1+\lceil2\eta\rceil)}}{(2b^{4})^{2+\lceil2\eta\rceil}} \cdot n^{2+\lceil2\eta\rceil}.
\]
If $\Delta > \frac{b^2n}{10e(4\eta+b^2)}$, letting $\kappa := 2+\lceil\frac{2\eta}{b}\rceil$,
 then we establish the theorem since 
\[
\frac{3(\lceil 2\eta \rceil + 2)^{4(1+\lceil2\eta\rceil)}}{(2b^{4})^{2+\lceil2\eta\rceil}} \cdot n^{2+\lceil2\eta\rceil} 
\le \frac{3(\lceil 2\eta \rceil + 2)^{4\kappa}}{(2b^{4})^{\kappa}} \cdot   \Big(\frac{10e(4\eta + b^2)}{b^2}\Big)^{\kappa} \cdot 
\Delta^\kappa
\le \frac{(240e)^{4\kappa} \cdot (\lceil \eta \rceil + 1)^{5\kappa} \cdot \Delta^\kappa}{b^{6\kappa}}.
\]
Thus, we assume $\Delta \le \frac{b^2n}{10e(4\eta + b^2)}$.
Let $V_1, \dots, V_k \subseteq V$ be disjoint independent sets such that $\bigcup_j V_j = V$.
 	We take $\theta = \frac{1}{5e\Delta}$
	so that $\frac{2}{n} \cdot(\frac{4\eta}{b^2} + 1) < \theta$.
	Let $S$ be a subset of vertices of size $\lceil\theta n\rceil$ chosen uniformly
 at random from all the subsets of size $\lceil\theta n\rceil$.
	Let $S_1, \dots, S_m \subseteq V$ be the connected components of $S$.
    Theorem \ref{thm:SI2UBF} implies that $\lceil\theta n\rceil$-UBF holds with constant
	\begin{equation}
	C_{\mathrm{UBF}}=\Big(\frac{e}{\theta}\Big)^{\lceil\frac{2\eta}{b}\rceil} = 
	\Big(5e^2\Delta\Big)^{\lceil\frac{2\eta}{b}\rceil},
	\end{equation}
    and so
	for any function $f : \Omega \rightarrow \mathbb{R}_{\ge0}$ we have 
	\begin{equation}
		\label{eq:proofstep1}
		\Ent_{\mu}(f) \le\Big(5e^2\Delta\Big)^{1+\lceil\frac{2\eta}{b}\rceil} \E_S\big[ \E_{\tau\sim \mu_{V\setminus S}}\left[\Ent_{S}^\tau (f^\tau) \big]\right],
	\end{equation}
 where ${\E}_S$ denotes the expectation over the random subset $S$.
 To bound the right-hand side of \eqref{eq:proofstep1}, we use the following fact, which we prove later in Section~\ref{subsect:sw:aux}.
 \begin{lemma}
	\label{lemma:proofstep3}
	Let $V_1, \dots, V_k$ be disjoint independent sets such that $\bigcup_{j=1}^k V_j = V$.
        Let $S\subseteq V$ be a subset of vertices.
	Let $S_1, \dots, S_m$ $\subseteq S$ be the connected components of the subgraph induced by $S$.
	Suppose that for $S_i \subseteq S$, 
	$\Gamma(S_i)$ takes the minimum value such that the following inequality holds 
	for an arbitrary pinning $\tau\in\Omega_{V\setminus S_i}$ and any function $g:\Omega_{S_i}^\tau\rightarrow \mathbb{R}_{\ge0}$:
	\begin{equation}
		\label{eq:proofstep3}
		\Ent^{\tau}_{S_i}(g)\le \Gamma(S_i) \sum_{j=1}^k \E_{\xi\sim \mu_{S_i \setminus V_j}^\tau}\left[\Ent_{V_j\cap S_i}^{\xi \cup \tau} (g^{\xi}_{S_i\cap V_j})\right].
	\end{equation}
	Then for any function $f : \Omega \rightarrow \mathbb{R}_{\ge0}$,
		\begin{equation}
			\label{eq:proofstep8}
			  \E_{\tau\sim \mu_{V\setminus S}}\left[\Ent_{S}^\tau (f^\tau) \right]
			 \le  \sum_{j=1}^k \E_{\tau\sim \mu_{V\setminus V_j}}\left[\Ent_{V_j}^\tau (f^\tau) \right] \cdot \max_{S_i \subseteq S} \Gamma(S_i).
		\end{equation}
\end{lemma}
\noindent
From \eqref{eq:proofstep1} and Lemma~\ref{lemma:proofstep3}, 
we have 
\begin{equation}
	\label{eq:proofstep7}
	\Ent_{\mu}(f) \le  \Big(5e^2\Delta\Big)^{\kappa} \sum_{j=1}^k \E_{\tau\sim \mu_{V\setminus V_j}}\left[\Ent_{V_j}^\tau (f^\tau) \right] \cdot \E_S\Big[\max_{S_i \subseteq S} \Gamma(S_i)\Big].
\end{equation}
	To show \eqref{eq:SI2KPF2}, it remains to provide an upper bound for
	$\E_S\left[\max_{S_i \subseteq S} \Gamma(S_i)\right]$.
        
    By assumption, 
	$\mu$ is $\eta$-spectrally independent and $b$-marginally bounded.
	These properties, by definition, are preserved under any pinning.
	In particular, 
	for any $S_i \subseteq S$ and an arbitrary pinning $\tau\in\Omega_{V\setminus S_i}$,
	$\mu_{S_i}^\tau$ is still $\eta$-spectrally independent and $b$-marginally bounded. 
    Hence, by Lemma~\ref{lemma:CGSVgap} and Lemma~\ref{lemma:gap2KPF}, we have 
    \[
    \Gamma(S_i) \le \frac{3(\lceil 2\eta \rceil + 2)^{4\kappa}}{(2b^{4})^{\kappa}} \cdot |S_i|^\kappa,
    \] and 
		\begin{equation}
			\label{eq:Gamma}
			\E_{S}\left[\max_{S_i \subseteq S} \Gamma(S_{i})\right]
			 \le b_1 \E_{S}\left[\max_{S_i \subseteq S} |S_{i}|^\kappa\right]
			= b_1\E_{S}\left[ \max_{v\in S} |C_{S}(v)|^{\kappa}\right],
		\end{equation}
  where $b_1:= \frac{3(\lceil 2\eta \rceil + 2)^{4\kappa}}{(2b^{4})^{\kappa}}$.
		To estimate the expectation on the right-hand side of \eqref{eq:Gamma},
		we first expand the expectation and apply a union bound as follows: 
  		\begin{align}
			\E_{S}\left[ \max_{v\in S} |C_{S}(v)|^{\kappa}\right] &= 
			\sum_{x=0}^{|S|} x^{\kappa} \cdot {\Pr}_{S}\left[ \max_{v\in S}  |C_{S}(v)| = x\right] \notag\\
			&\le 	\nonumber		(2\log_2|S|)^{\kappa} + \sum_{x=2\log_2|S|}^{|S|} x^{\kappa} \cdot  {\Pr}_{S}\left[ \max_{v\in S}  |C_{S}(v)| = x\right] \\
			& \le \label{eq:proofstep15.1} 	(2\log_2|S|)^{\kappa} + \sum_{x=2\log_2|S|}^{|S|} x^{\kappa} \cdot \sum_{v\in S} {\Pr}_{S}\left[  |C_{S}(v)| = x\right]. 
		\end{align}
  Then, applying
Lemma~\ref{lemma:g1} and noting that $\theta < 1/(4e\Delta)$, we obtain
\begin{align}
 \sum_{x=2\log_2|S|}^{|S|} x^{\kappa} &\cdot \sum_{v\in S} {\Pr}_{S}\left[  |C_{S}(v)| = x\right]
    	\le \nonumber \lceil \theta n \rceil\sum_{x=2\log_2|S|}^{|S|} x^{\kappa} (2e\Delta \theta)^{x-1} \\
			& = \nonumber\frac{ \lceil\theta n \rceil}{2e\Delta \theta} \cdot (2e\Delta \theta)^{2\log_2 |S|}  \sum_{x=2\log_2|S|}^{|S|} x^{\kappa} (2e\Delta \theta)^{x - 2\log_2|S|} \\
			& \le  \nonumber  \frac{1}{2|S|e\Delta}  \sum_{x=0}^{|S|-2\log_2|S|} (x+2\log_2|S|)^{\kappa} 2^{-x}\\
   & \le \nonumber \frac{1}{2|S|e\Delta}  
   \left[ 
    \sum_{x=0}^{\log_2|S| - 1} (x + 2\log_2 |S|)^{\kappa} 
    + \sum_{x=\log_2 |S|}^{|S|-2\log_2|S|} \frac{(x+2\log_2|S|)^{\kappa}}{|S| \cdot 2^{x - \log_2 |S|}}
   \right] \\
   & \le    \frac{1}{2|S|e\Delta}  
   \left[ 
   \label{eq:proofstep15.2} (3\log_2|S|)^{\kappa} + 
    \sum_{x=0}^{|S|-3\log_2|S|} \frac{(x+3\log_2|S|)^{\kappa}}{|S| \cdot 2^x}
   \right].
\end{align}
When $|S|=\omega(1)$, $(3\log_2|S|)^{1+\kappa}/|S| < 1$. 
Also, for any integer $x\ge 0$, $\frac{(x+3\log_2|S|)^{\kappa}}{|S| \cdot 2^x} < 1$,
so
the last sum in \eqref{eq:proofstep15.2} is less than $|S|$.
		Therefore, by \eqref{eq:Gamma}, \eqref{eq:proofstep15.1} and  \eqref{eq:proofstep15.2} we have 
		\begin{equation}
            \label{eq:proofstep14}
			\E_{S}\left[\max_{S_i \subseteq S} \Gamma(S_{i})\right]  \le b_1 \cdot [ (2\log_2|S|)^{\kappa} + 1 ] 
		\le b_1  (3\log_2|S|)^{\kappa}.
		\end{equation}
		These bounds together with \eqref{eq:proofstep7} imply that  
	\[
            C_{\mathrm{KPF}}\le b_1 (3\log_2 n)^{\kappa} \cdot\Big(5e^2\Delta\Big)^{\kappa}
             = 3\cdot \big(\frac{15e^2}{2}\big)^\kappa \cdot \frac{(\lceil 2\eta\rceil +2)^{4\kappa}}{b^{4\kappa}} \cdot (\Delta \log_2 n)^\kappa,
    \]
	establishing the desired bound. When $|S|=O(1)$, the left-hand side of~\eqref{eq:proofstep14} can be bounded by an absolute constant, and the result follows from~\eqref{eq:proofstep7}.
\end{proof}

\subsection{Entropy factorization: Proof of Lemma~\ref{lemma:proofstep3}}
\label{subsect:sw:aux}
We proceed with the proof of  Lemma~\ref{lemma:proofstep3} by 
first presenting several facts that will be useful.

\begin{lemma}[Lemma 2.7, \cite{BCCPSV22}]
	\label{lemma:cprent}
	Let $\Lambda = A\cup B \subseteq V$,  $\tau \in \Omega_{V\setminus \Lambda}$, and assume $\mu_\Lambda^{\tau}$ is a product measure $\mu_\Lambda^{\tau} = \mu_A^{\tau} \otimes \mu_B^{\tau}$. 
	For all $U\subset B$ and any $f: \Omega \rightarrow \mathbb{R}_{\ge0}$,
	\begin{enumerate}
 		\item	$		\Ent^\tau_{A}(f^\tau_{A}) = \E_{\gamma\sim \mu^\tau_B} \left[\Ent^{\gamma\cup \tau}_A(f^\tau_{A})\right].$
		\item $		\E_{\gamma\sim \mu^\tau_B} \left[\Ent^{\gamma\cup\tau}_A(f_{A}^{\tau})\right]
		\le 	\E_{\gamma\sim \mu^\tau_U} \left[\Ent^{\gamma \cup \tau}_A(f^{\gamma\cup\tau}_A)\right].$
	\end{enumerate}
\end{lemma}
\begin{lemma}[Lemma 3.1, \cite{CP20}]
	\label{lemma:lote}
	Let $\Lambda_0 = \emptyset$.
	For any $\Lambda_1 \subset \dots \Lambda_m \subset \Lambda \subseteq V$, any $\tau\in \Omega_{V\setminus \Lambda}$ and any $f : \Omega_\Lambda^\tau \rightarrow \mathbb{R}_{\ge0}$,
$$
	\sum_{i=1}^{m} \E_{\gamma \sim \mu_{\Lambda\setminus \Lambda_i}^\tau} \left[ \Ent^{\tau \cup \gamma}_{\Lambda_i\setminus \Lambda_{i-1}}(f^{\gamma}_{\Lambda_i \setminus \Lambda_{i-1}}) \right]
	= \E_{\gamma \sim \mu^\tau_{\Lambda \setminus \Lambda_m}} \left[ \Ent^{\tau \cup \gamma}_{\Lambda_m} (f^{\gamma}) \right].
$$
\end{lemma}
\noindent
The following corollary directly follows from this fact,
by taking $\Lambda_1 = A, \Lambda_2 = B$ and $m=2$.
\begin{corollary}\label{cor:monent}
	Let $A,B$ and $\Lambda$ be subsets of vertices such that $A\subset B \subset \Lambda \subseteq V$. For any $\tau\in \Omega_{V\setminus \Lambda}$ and any $f : \Omega_\Lambda^\tau \rightarrow \mathbb{R}_{\ge0}$,
	$$
	\E_{\gamma \sim \mu^\tau_{\Lambda \setminus A}} \left[ \Ent^{\gamma \cup \tau}_{A} (f^{\gamma}) \right]
	\le \E_{\gamma \sim \mu^\tau_{\Lambda \setminus B}} \left[ \Ent^{\gamma \cup \tau}_{B} (f^{\gamma}) \right].
	$$
\end{corollary}
\noindent
We are now ready to prove Lemma~\ref{lemma:proofstep3}.
\begin{proof}[Proof of Lemma~\ref{lemma:proofstep3}]
	Note that $\mu_S^\tau = \otimes_{i=1}^m \mu_{S_i}^\tau$ is a product measure.
	For $i\ge 1$, let $S_{\le i} := S_1\cup \dots \cup S_i$. 
	For $i > 1$, we let $S_{<i}:=S_1 \cup \dots \cup S_{i-1}$,
	and we set $S_{<1} := \emptyset$ for convenience.
As a direct consequence of applying Lemma~\ref{lemma:lote} and  applying Lemma~\ref{lemma:cprent}(1), we have the following identity for any $f : \Omega \rightarrow \mathbb{R}_{\ge0}$:
\begin{equation}
	\label{eq:proofstep2}
	\E_{\tau\sim \mu_{V\setminus S}}\left[\Ent_{S}^\tau (f^\tau) \right]
        = \sum_{i=1}^{m} \E_{\tau \sim \mu_{V\setminus (S_{\le i})}} 
	\left[  \Ent^{\tau}_{S_i}(f^{\tau}_{S_i}) 
	\right]
	=\sum_{i=1}^{m} \E_{\tau \sim \mu_{V\setminus (S_{\le i})}} 
	\left[ 
		\E_{\gamma \sim \mu^{\tau}_{S_{<i}}} \big[ \Ent^{\tau\cup \gamma}_{S_i}(f^{\tau}_{S_i}) \big]
	\right].
\end{equation}
On the other hand, setting $g=f_{S_i}^{\tau}$ in \eqref{eq:proofstep3}, then for any $\gamma \in \Omega^\tau_{S_{<i}}$ we obtain that 
\begin{equation}
	\label{eq:proofstep9}
	\Ent^{\tau \cup \gamma}_{S_i}(f_{S_i}^{\tau})\le \Gamma(S_i) \sum_{j=1}^k \E_{\xi\sim \mu_{S_i \setminus V_j}^{\gamma \cup\tau}}\left[\Ent_{V_j\cap S_i}^{\xi \cup \tau \cup \gamma} (f_{S_i \cap V_j}^{\tau \cup \xi})\right].
\end{equation}
Combining \eqref{eq:proofstep2} and \eqref{eq:proofstep9} yields
\begin{align}
	\label{eq:proofstep4}
	\E_{\tau\sim \mu_{V\setminus S}}\left[\Ent_{S}^\tau (f^\tau) \right] &\le
		\sum_{i=1}^{m} \E_{\tau \sim \mu_{V\setminus S_{\le i}}} 
		\left[ 
			\E_{\gamma \sim \mu^{\tau}_{S_{<i}}} 
			\Big[ 
				\Gamma(S_i) \sum_{j=1}^k \E_{\xi\sim \mu_{S_i \setminus V_j}^{\tau \cup \gamma}}\big[\Ent_{V_j\cap S_i}^{\xi\cup \gamma \cup \tau} (f_{S_i \cap V_j}^{\xi\cup \tau})\big] 
			\Big]
		\right]\nonumber\\
	&= \sum_{j=1}^k 
		\sum_{i=1}^m\Gamma(S_i)\E_{\tau \sim \mu_{V\setminus S_{\le i}}} 
                \E_{\xi \sim \mu^{\tau}_{S_i \setminus V_j}} 
			\E_{\gamma\sim \mu_{S_{<i}}^{\tau \cup \xi}}
                    \left[\Ent_{V_j\cap S_i}^{\xi\cup \gamma \cup \tau} (f^{\xi\cup \tau}_{S_i \cap V_j})\right] 
	\nonumber\\
	&\le \sum_{j=1}^k \max_i \Gamma(S_i)
		\sum_{i=1}^m\E_{\tau \sim \mu_{(V\setminus S_{\le i}) \cup (S_i \setminus V_j)}} 
			\E_{\gamma \sim \mu^{\tau }_{S_{<i}}} 
			\left[\Ent_{V_j\cap S_i}^{ \gamma \cup \tau} (f^{ \tau}_{S_i \cap V_j})\right]. 
	\end{align}
We show next that for any $j=1,\dots, k$, the following inequality holds: 
\begin{equation}
	\label{eq:proofstep5}
	\sum_{i=1}^m\E_{\tau \sim \mu_{(V\setminus S_{\le i}) \cup (S_i \setminus V_j)}} 
	\E_{\gamma \sim \mu^{\tau }_{S_{<i}}} 
	\left[\Ent_{V_j\cap S_i}^{ \gamma \cup \tau} (f^{ \tau}_{S_i \cap V_j})\right]
	\le \E_{\tau\sim \mu_{V\setminus V_j}}\left[\Ent_{V_j}^\tau (f^\tau) \right].
\end{equation}
Given a pinning $\tau\sim \mu_{V\setminus ((S_i \cap V_j) \cup S_{<i})}$,
$\mu_{S_i\cap V_j}$ and $\mu_{S_{<i}}$ are independent. By applying  Lemma~\ref{lemma:cprent}(2) to $\E_{\gamma \sim \mu^{\tau }_{S_{<i}}} 
	[\Ent_{V_j\cap S_i}^{ \gamma \cup \tau} (f^{ \tau}_{S_i \cap V_j})]$, 
we have
\begin{equation}
    \label{eq:proofstep5.1}
	\sum_{i=1}^m\E_{\tau \sim \mu_{(V\setminus S_{\le i}) \cup (S_i \setminus V_j)}} 
	\E_{\gamma \sim \mu^{\tau }_{S_{<i}}} 
	\left[\Ent_{V_j\cap S_i}^{ \gamma \cup \tau} (f^{ \tau}_{S_i \cap V_j})\right]
	\le \sum_{i=1}^m 
	\E_{\tau \sim \mu_{(V\setminus S_{\le i}) \cup (S_i \setminus V_j)}} \E_{\xi \sim \mu_{(S_{<i}) \setminus V_j}^{\tau}}
		\left[\Ent_{S_i\cap V_j}^{\tau \cup \xi} (f_{S_i\cap V_j}^{\tau \cup \xi}) \right].
\end{equation}
Letting $\phi = \tau \cup \xi$, and by applying Lemma~\ref{lemma:cprent}(1) to $\Ent_{S_i\cap V_j}^{\phi} (f_{S_i\cap V_j}^{ \phi})$ we also have 
\begin{equation}\label{eq:proofstep5.2}
    \sum_{i=1}^m
    \E_{\phi \sim \mu_{V\setminus((S_{\le i}) \cap V_j)}}
        \left[\Ent_{S_i\cap V_j}^{\phi} (f_{S_i\cap V_j}^{ \phi}) \right]
=\sum_{i=1}^m
   \E_{\phi \sim \mu_{V\setminus(S_{\le i}) \cap V_j}}
   \E_{\psi \sim \mu^{\phi}_{(S_{<i}) \cap V_j}}
        \left[\Ent_{S_i\cap V_j}^{\phi \cup \psi} (f_{S_i\cap V_j}^{\phi}) \right].
\end{equation}
Also, the following identity follows from Lemma~\ref{lemma:cprent}(1) and Lemma~\ref{lemma:lote} as in the way of obtaining \eqref{eq:proofstep2}: 
\begin{equation}
    \label{eq:proofstep5.3}
	\E_{\tau \sim \mu_{V \setminus (S\cap V_j)}} \left[\Ent_{S \cap V_j}^\tau (f^{\tau})\right] = \sum_{i=1}^m\E_{\phi \sim \mu_{V\setminus(S_{\le i} \cap V_j)}}
		\E_{\psi \sim \mu^{\phi}_{(S_{<i}) \cap V_j}}
			\left[\Ent_{S_i\cap V_j}^{\phi \cup \psi} (f_{S_i\cap V_j}^{\phi}) \right].
\end{equation}
Finally, it follows from Corollary~\ref{cor:monent} that
\begin{equation}
    \E_{\tau \sim \mu_{V \setminus (S\cap V_j)}} \left[\Ent_{S \cap V_j}^\tau (f^{\tau})\right] \le 
    \label{eq:proofstep5.4} \E_{\tau \sim \mu_{V \setminus V_j}} \left[\Ent_{V_j}^\tau (f^{\tau})\right],
\end{equation}
so \eqref{eq:proofstep5} follows from \eqref{eq:proofstep5.1},
\eqref{eq:proofstep5.2}, \eqref{eq:proofstep5.3} and \eqref{eq:proofstep5.4}.
Therefore, we obtain \eqref{eq:proofstep8} by 
\eqref{eq:proofstep4} and \eqref{eq:proofstep5}.
\end{proof}

\subsection{Applications of Theorem~\ref{thm:SW}}
\label{subsec:sw:app}

In this section, we prove Corollary~\ref{cor:SW:potts:intro} from the introduction 
and present another application of Theorem~\ref{thm:SW} concerning the SW dynamics on a random graph generated from the classical Erd\H{o}s-R\'enyi $G(n,p) $ model. For this, we first define Dobrushin's influence matrix.

\begin{definition}
    The \emph{Dobrushin influence matrix} $A\in \mathbb{R}^{n \times n}$ is defined by $A(u,u)=0$ and for $u\neq v$,
\[
A(u,v) = \max_{ (\sigma, \tau)\in S_{u,v}}d_{TV}(\mu_{v}(\cdot \mid \sigma), \mu_v(\cdot \mid \tau)), 
\] 
where \emph{$S_{u,v}$} contains the set of all pairs of partial configurations $(\sigma, \tau)$ in $\Omega_{V\setminus \{v\}}$ that can only disagree at $u$, namely, 
$\sigma_w = \tau_w$ if $w\neq u$. 
\end{definition}

It is known that an upper bound on the spectral norm of $A$ implies spectral independence. In particular, we have the following result from \cite{BCCPSV22}.
\begin{proposition}
	[Theorem 1.13, \cite{BCCPSV22}] \label{prop:D2SI}
	If the Dobrushin influence matrix $A$ of a distribution $\mu$ satisfies $\|A\| \le 1 - \epsilon$ for some $\epsilon > 0$,
	then $\mu$ is spectral independent with constant $\eta= 2/\epsilon$.
\end{proposition}

For the ferromagnetic Ising model, $\beta_u(\Delta):=\ln \frac{\Delta}{\Delta-2}$ corresponds to the threshold value of the parameter $\beta$ for the uniqueness/non-uniqueness phase transition on the $\Delta$-regular tree.
For the anti-ferromagnetic Ising model, the phase transition occurs at $\bar{\beta}_u(\Delta):=-\ln \frac{\Delta}{\Delta-2}$.
If $\bar{\beta}_u(\Delta)(1 - \delta)<\beta <\beta_u(\Delta)(1 - \delta)$, we say the Ising model satisfies the \emph{$\delta$-uniqueness condition}. 
On a bounded degree graph, 
$\|A\| \le 1 - \delta$ for the Ising model is a strictly stronger condition than 
$\delta$-uniqueness condition. However, due to the observation made in \cite{AJKPV22}, 
if $\Delta\rightarrow \infty$, the two conditions are roughly equivalent. 
\begin{proposition}
	\label{prop:Dobrushin}
	The  Ising model with parameter $\bar{\beta}_u(\Delta)(1 - \delta)<\beta <\beta_u(\Delta)(1 - \delta)$
	and $\Delta \rightarrow \infty$
	satisfies 
 $\| A\| \le 1- \delta/2$. 
\end{proposition}

\begin{proof}
	We verify that the Ising model 
 has bounded spectral norm of $A$:
note that each entry of $A$ can be upper bounded by $|\beta|/2$ \cite{Hayes06}, so a row sum of $A$ is at most
	\[ \frac{|\beta|\Delta}{2}
	 < \frac{(1-\delta)\Delta}{2} \ln \left( 1 + \frac{2}{\Delta - 2}\right) \le 
	\frac{(1-\delta)\Delta}{2}\left( \frac{2}{\Delta - 2}\right)
	= (1-\delta)\left(1+ \frac{2}{\Delta-2}\right) < 1-\delta/2,
	\]
	where the last inequality holds for $\Delta$ large enough.
\end{proof}

We show next that Corollary~\ref{cor:SW:potts:intro} indeed follows from Theorem~\ref{thm:SW}.
For this, we first restate the corollary in a more precise manner.
\begin{corollary}
	\label{cor:SW}
	Let $\delta\in(0,1)$ and $\Delta \ge 3$.
	For the ferromagnetic Ising model with $\beta \le (1-\delta)\beta_u(\Delta)$ 
	on any graph $G$ of maximum degree $\Delta$ and chromatic number $\chi$, 
	or for the ferromagnetic $q$-state Potts model with $q\ge 3$ and 
	$\beta \le \frac{2(1-\delta)}{\Delta}$ on the same graph, 
	the mixing time of the SW dynamics satisfies
 \[T_{mix}(P_{SW}) =O\big( \chi \cdot \Delta^{\kappa} \cdot (\log n)^{1+\kappa} \big),\]
 where $\kappa=2 + \lceil \frac{4qe^2}{\delta}\rceil$.
\end{corollary}

\begin{proof}
	If $\Delta=O(1)$,
	then the corollary was proved in a stronger form in \cite{BCCPSV22}. 
	Thus, we assume $\Delta \rightarrow \infty$.
 
We first show spectral independence.   
Let $q=2$. 
Under the $\delta$-uniqueness condition $0<\beta <(1-\delta)\beta_u(\Delta)$, 
by Proposition~\ref{prop:Dobrushin} and Proposition~\ref{prop:D2SI}, 
the Ising model $\ising$ satisfies $(4/\delta)$-spectral independence.
For the $q$-state Potts model with $q\ge 3$, 
the Dobrushin influence matrix corresponding to $\potts$ satisfies
$\|A\| \le \frac{1}{2}\beta\Delta$; see proof of Theorem 2.13 in \cite{Ullrich14}.
Thus, if $\beta \le \frac{2(1-\delta)}{\Delta}$, then 
$\|A\| \le 1-\delta$, and by Proposition~\ref{prop:D2SI}, 
$\potts$ satisfies $(2/\delta)$-spectral independence.

Letting $N(v)$ denote the neighborhood of $v$, and noting that 
for any configuration $\eta$ on $N(v)$ we have
$\mu(\sigma_v = c \mid \sigma_{N(v)} = \eta) \ge 1/(qe^2)$,
we deduce that $\potts$ and $\ising$ are both $(1/(qe^2))$-marginally bounded.
Therefore, by noting that $\kappa = 2 + \lceil \frac{4qe^2}{\delta}\rceil$ is a constant that only depends on $\delta$, 
the mixing time bound follows from Theorem~\ref{thm:SW}
\[T_{mix}(P_{SW}) = O\Big(  \chi \cdot (C \Delta\log n)^{\kappa}   \cdot 
(qe^2)^{(2+6\kappa)} (1+4/\delta)^{5\kappa} \cdot 
\log n \Big), 
\]
as desired.
\end{proof}

\subsubsection{The SW dynamics on random graphs}
\label{subsec:rg}

As another application of Theorem~\ref{thm:SW}, 
we consider the SW dynamics on a random graph 
generated from the classical $G(n, \frac{d}{n})$ model in which each edge is included independently with probability $p = d/n$; 
we consider the case where $d$ is a constant independent of $n$.
In this setting, while
a typical graph has $\tilde{O}(n)$ edges, 
its maximum degree is of order $\Theta(\frac{\log n}{\log \log n})$ with high probability.
Our results imply that the SW dynamics has polylogarithmic mixing on this type of graph provided $\beta$ is small enough.

\begin{corollary}
	\label{cor:SWrg}
	Let $\delta\in(0,1)$ and $d\in\mathbb{R}_{\ge0}$ be constants independent of $n$.
	Suppose that $G\sim G(n,d/n)$ and $G$ has maximum degree $\Delta$.
	For the ferromagnetic Ising model with parameter $\beta < (1-\delta)\beta_u(\Delta)$ on $G$ 
	or the ferromagnetic $q$-Potts model with $q \ge 2$ and
	$\beta \le \frac{2(1-\delta)}{\Delta}$
	on the same graph, 
the SW dynamics has $O\Big( (\log n)^{5+ 2\lceil \frac{4qe^2}{\delta}\rceil} \Big)$ mixing time,	
		with high probability over the choice of the random graph $G$.
\end{corollary}
\noindent
Corollary~\ref{cor:SWrg} is established using Corollary~\ref{cor:SW} and the following fact about random graphs.
\begin{proposition} [\cite{AN07}]
	\label{prop:color}
	Let $G\sim G(n, \frac{d}{n})$ for a fixed $d\in\mathbb{R}_{\ge0}$, 
	and let $\chi$ be the chromatic number of $G$.
	With high probability over the choice of $G$, $\chi = k_d$ or $\chi = k_d +1$, where 
	$k_d$ is the smallest integer $k$ such that $d < 2k \log k$.
\end{proposition}
\begin{proof}[Proof of Corollary~\ref{cor:SWrg}]
	By Proposition~\ref{prop:color}, with high probability $G\sim G(n,\frac{d}{n})$ has 
	chromatic number $\chi = O(d)$.
	Also, it is known that with high probability $\Delta = \Theta(\frac{\log n}{\log \log n})$.
	Suppose both properties hold.
	The result follows from Corollary~\ref{cor:SW}.
\end{proof}

\section{Systematic scan dynamics}
\label{sec:ss}
In this section, we study the systematic scan dynamics for general spin systems (see Definition \ref{def:spin-system}), which we define next.
Given an ordering $\phi = [v_1, \dots, v_n]$ of the vertices, a systematic scan dynamics performs heat-bath updates on $v_1, \dots, v_n$ 
sequentially in this order.  
Recall that a heat-bath update on $v_i$ simply means the replacement of the spin on $v_i$ by a new spin assignment generated according to 
the conditional distribution in $v_i$ given the configuration in $V\setminus \{v_i\}$.
Let $P_i \in \mathbb{R}^{|\Omega| \times |\Omega|}$ be the transition matrix corresponding to a heat-bath update on the vertex $v_i$. 
The transition matrix of the systematic scan dynamics for the ordering $\phi$ can be written as 
$\mathcal{S}_{\phi} := P_n \dots P_1$.
In general, $\mathcal{S}_{\phi}$ is not reversible, so as in earlier works we work with the symmetrized version of the scan dynamics
that updates the spins in the order $\phi$ and in addition updates the spins in the reverse order of $\phi$ \cite{Fill91,MT06}. 
The transition matrix of the symmetrized systematic scan dynamics can then be written as
\[
P_{\phi} := \prod_{i=1}^{n} P_i \prod_{i=0}^{n-1} P_{n-i}.
\]
Henceforth, we only consider the symmetrized version of the dynamics.
Since $P_\phi$ is a symmetrized product of reversible transition matrices, 
one can straightforwardly verify its reversibility with respect to $\mu$; its ergodicity follows from the assumption that the spin system is totally-connected (see Definition \ref{def:tc}). 

We show tight mixing time bounds for $P_\phi$ for monotone spin systems (see Definition~\ref{def:monotone}). Our main result for the systematic scan dynamics is Theorem~\ref{thm:SIIsing:intro} from the introduction, which we restate here for convenience. The proof of this theorem is provided in Section~\ref{sec:ssthm}.

\ssintro*

\noindent
We complement Theorem~\ref{thm:SIIsing:intro} with a lower bound for the mixing time of systematic scan  dynamics for a particular ordering $\phi$. Specifically, on a bipartite graph $G=(V_E \cup V_O, E)$, an \emph{even-odd scan dynamics $P_{EOE}$} is a systematic scan dynamics with respect to an ordering $\phi$ such that $v_e$ appears before $v_o$ in $\phi$ for all $v_e\in V_E$ and $v_o \in V_O$. In other words, 
\[
P_{\phi} = \prod_{i:v_i\in V_E} P_i \prod_{i:v_i\in V_O} P_i \prod_{i:v_i\in V_O} P_i  \prod_{i:v_i\in V_E} P_i.
\]
The above expression is well-defined without specifying the ordering in which the vertices in $V_E$ and $V_O$ are updated since the updates commute. 
\begin{lemma}
    \label{lemma:lbss}
    Let $\Delta$ be a constant and let $G$ be an $n$-vertex connected bipartite graph with maximum degree $\Delta$.
    The even-odd scan dynamics $P_{EOE}$ for the ferromagnetic Ising model on $G$ has mixing time
    $
        T_{mix}(P_{EOE}) = \Omega(\log n).
    $
\end{lemma}

The lower bound in Lemma~\ref{lemma:lbss} 
is proved in Section~\ref{sec:lb} using the machinery from~\cite{HS07} and the fact 
that even-odd scan dynamics does not propagate disagreements quickly (under a standard coupling).
Our proof can thus be extended to other scan orderings that propagate disagreements slowly; however, there are orderings that do propagate disagreements quickly (think of a box in $\mathbb{Z}^2$ with the vertices sorted in a ``spiral'' from the boundary of the box to its center). For this type of ordering, the technique does not provide the $\Omega(\log n)$ lower bound.
In addition, while we focus on the ferromagnetic Ising model to ensure clarity in the proof, the established lower bound is expected to apply to a broader class of spin systems.

\subsection{Proof of main result for systematic scan dynamics: Theorem~\ref{thm:SIIsing:intro}}
\label{sec:ssthm}

The main technique in the proof of Theorem~\ref{thm:SIIsing:intro} is to compare the systematic scan dynamics with a fast mixing block dynamics 
via a censoring inequality developed in \cite{FillKahn13}.
For this, we first introduce some notations and definitions.

We start by reviewing standard facts about the coupling method that will be used in our proofs; see \cite{LevinPeresWilmer2006} for a more detailed background.
A \emph{coupling} of a Markov chain $M$ specifies, for every pair of states $(X_t,Y_t)\in \Omega \times \Omega$ at every step $t$, 
a probability distribution $\Pp$ over $(X_{t+1}, Y_{t+1})$ such that when viewed in isolation, 
$\{X_t\}$ and $\{Y_t\}$ are valid instances of the chain $M$.
The \emph{optimal coupling lemma} says that for any two distributions
$\mu$ and $\nu$, we have
	\begin{equation} \label{eq:optcoupling}
	    	\| \mu -\nu \|_{TV} = \inf_{X\sim \mu, Y\sim \nu} \Pp\left[X\neq Y: (X,Y) \text{ is a coupling of } \mu \text{ and } \nu\right],
	\end{equation}
 where the infimum is taken over all couplings of $\mu$ and $\nu$.
We focus on couplings of Markov chains such that  if $X_s = Y_s$ then $X_t = Y_t$ for all $t\ge s$.
Given a coupling $\Pp$ of $M$, the \emph{coupling time}, is defined as 
\[\Tcoup(M) :=  \min_{T>0} \Big\{ \max_{X_0 \in \Omega, Y_0\in \Omega} \Pp[X_T \neq Y_T]\le \frac{1}{4} \Big\}.\] 
It is a standard fact that for any coupling $(X_t, Y_t)$, the coupling time bounds the mixing time as follows:
 \begin{equation}
 \label{eq:coupling}
     d(t) \le \max_{X_0 \in \Omega, Y_0\in \Omega} \Pp[X_T \neq Y_T], \text{ and thus } T_{mix}(M) \le \Tcoup(M).
 \end{equation}
\noindent
A coupling of two instances $\{X_t\}, \{Y_t\}$ of a Markov chain $M$ is a \emph{monotone coupling} if 
$X_{t+1} \ge_{q} Y_{t+1}$ whenever $X_t \ge_{q} Y_t$, where $\ge_{q}$ is the partial ordering of $\Omega$. 
Let $\{X_{t, \sigma}\}$ denote the instance of $M$ starting at configuration $\sigma \in \Omega$.
If there exists a simultaneous monotone coupling of $\{X_{t, \sigma}\}$ for all $\sigma \in \Omega$
(i.e., a grand coupling),
then we say $M$ is a \emph{monotone Markov chain}.
It can be checked that $P_{\phi}$ is a monotone Markov chain for any $\phi$ (see e.g. \cite{BCV18}). 

We may also define a partial ordering $\preceq_{\pi}$ on the space of transition matrices.
A function $f\in \mathbb{R}^{|\Omega|}$ is said to be \emph{non-decreasing} if $f(\sigma) \ge f(\tau)$ whenever $\sigma \ge_{q} \tau$,
or \emph{non-increasing} if $f(\sigma) \le f(\tau)$ whenever $\sigma \ge_{q} \tau$.
We endow $\mathbb{R}^{|\Omega|}$ with the inner product 
$\langle f, g\rangle_{\pi}:=\sum_{x\in \Omega} f(x)g(x) \pi(x)$, 
which induces a Hilbert space 
$(\mathbb{R}^{|\Omega|}, \langle \cdot, \cdot\rangle_{\pi})$ denoted as $L_2(\pi)$.
For transition matrices $K$ and $L$ whose stationary distributions are both $\pi$, 
we say $K\preceq_{\pi} L$ if $\langle Kf, g\rangle_{\pi} \le \langle Lf, g\rangle_{\pi}$ 
for every non-negative and non-decreasing functions $f,g\in L_2(\pi)$. 
To show $K\preceq_{\pi} L$ in our applications, we use the following facts.
\begin{proposition}[\cite{FillKahn13}]
    \label{prop:po}
    Suppose $\pi$ is the Gibbs distribution of a monotone spin system.     
	\begin{enumerate}
		\item If $A_1 \preceq_{\pi} B_1$ and $A_2 \preceq_{\pi} B_2$, then for $0\le \lambda \le 1$, 
		$(1-\lambda)A_1 + \lambda A_2 \preceq_{\pi} (1-\lambda) A_1 +\lambda A_2$.
		\item If $A_s \preceq_{\pi} B_s$ for $s=1,\dots,l$, then $A_1 \dots A_l \preceq_{\pi} B_1 \dots B_l$.
        \item For any fixed $v$, let $K_v$ be the heat-bath update at site $v$. 
            Then, $K_v \preceq_{\pi} I$.
	\end{enumerate}
\end{proposition}

Establishing such partial order between two transition matrices is significant as it would imply stochastic domination of the corresponding two chains (recall that for two distributions $\pi$ and $\nu$ on $\Omega$, we say $\pi$ \emph{stochastically dominates} $\nu$, and denote as $\pi \succeq \nu$, if
for any non-decreasing function $f\in \mathbb{R}^{|\Omega|}$, we have $ \E_{\pi}[f] \ge \E_{\nu}[f]$).
The following lemma captures such implication. 
\begin{lemma}[\cite{FillKahn13,BCV18}]
	\label{lemma:sd}
	Suppose $\{X_t\}$ and $\{Y_t\}$ are monotone ergodic Markov chains 
	reversible with respect to $\pi$, the Gibbs distribution of a monotone spin system. 	
	Let $K_X$ and $K_Y$ be the corresponding transition matrices of $\{X_t\}$ and $\{Y_t\}$. 
	Suppose $K_X\preceq_{\pi} K_Y$.
	Then $X_t \preceq Y_t$ for all $t\ge 0$ if 
	the initial states $X_0$ and $Y_0$ are sampled from a common distribution $\nu$ such that $\nu/\pi$ is non-decreasing; 
	if $\nu/\pi$ is instead non-increasing, then 
	$Y_t \preceq X_t$ for all $t\ge 0$, where $\preceq$ as a relation for $X_t$ and $Y_t$ denotes stochastic domination of their corresponding distributions at time $t$.
\end{lemma}

We now provide our proof of Theorem~\ref{thm:SIIsing:intro}.
\begin{proof}[Proof of Theorem~\ref{thm:SIIsing:intro}]
	We partition $V$ into $k$ disjoint independent sets $I_1, I_2, \dots, I_{k}$,
	where $k = O(\Delta)$. 
	Set $\mathcal{B} = \{I_1, \dots, I_k\}$ and define $P_\mathcal{B}$ to be the heat-bath block dynamics w.r.t. 
	these independent sets.
	Fix an ordering $\phi = [v_1, \dots, v_n]$, and fix $j \in \{1,\dots, k\}$.
	Let $K_{j}$ be the transition matrix corresponding to heat-bath update in the independent set $I_j$, which can also be seen as a systematic scan on $I_j$ 
	according to the ordering defined by $\phi$.  
	We define $\hat{P}_i$ to be $P_i$ if $i \in I_j$ and the identity matrix $I$ otherwise so that 
	\[
	K_j = K_j^2 
	= \left(\prod_{i: v_i \in I_j} P_i\right)^2 
	= \left(\prod_{i: v_i \in I_j} P_i \prod_{i: v_i \notin I_j} I\right)^2
	=  \prod_{i=1}^{n} \hat{P}_i \prod_{i=0}^{n-1} \hat{P}_{n-i}.
	\]
	Note that in the computation above, $P_i$ and $P_{i'}$ commute for $v_i,v_{i'}\in I_j$,
	and $I$ commutes with arbitrary matrices. 
	By Proposition~\ref{prop:po}(3), we obtain 
	$P_i \preceq_{\mu} \hat{P}_i$ for all $i$, and hence by Proposition~\ref{prop:po}(2), 
	we obtain $ P_\phi \preceq_{\mu} K_j$ for any $j$, and consequently, by Proposition~\ref{prop:po}(1),
	\begin{equation}
		\label{eq:PBandPphi}
		 P_\phi \preceq_{\mu}  \frac{1}{k}\sum_{j=1}^k K_j = P_{\mathcal{B}}.
	\end{equation}
	Let $+$ and $-$ denote the top and the bottom elements in $[q]$ respectively.
	Let $\{X^+_t\}$ (resp., $\{X^-_t\}$) be an instance of a Markov chain with transition matrix $P_\phi$ starting from the all $+$ (resp., all $-$) configuration.
	Similarly, let $\{Y^+_t\}$ (resp., $\{Y^-_t\}$) be an instance of $P_\mathcal{B}$
	starting from the all $+$ (resp., all $-$) configuration.
	$P_\phi$ is monotone, 
	so we can define a grand monotone coupling of $\{X^+_t\}$ and $\{X^-_t\}$
	such that $X^-_t \le_{q} X^+_t$ for all $t\ge 0$, 
	which with \eqref{eq:coupling} further implies that the mixing time of a systematic scan can be upper bounded by the coupling time of the all $+$ and all $-$ configurations.
 
	Letting $\nu^+$ and (resp., $\nu^-$) denote the trivial distribution concentrated on the all $+$ (resp., all $-$) configuration, we note that $\nu^+/\mu$ is non-decreasing and $\nu^-/\mu$ is non-increasing.
	Then Lemma~\ref{lemma:sd} and \eqref{eq:PBandPphi} imply that
	for all $t\ge 0$,
	\[
	Y_t^- \preceq X_t^- \preceq X_t^+ \preceq Y_t^+.
	\]
	 For any $v \in V$ and all $t\ge 0$, 
  $X^-_t \le_{q} X^+_t$ implies that 
	\begin{align*}
		\Pr[X_t^+(v) \neq X_t^-(v)] 
		&\le  \sum_{c\in [q]}  \Pr[X_t^+(v) \ge c, X_t^-(v) < c] \\
  	&= \sum_{c\in [q]}  \Pr[X_t^+(v) \ge c] - \Pr[X_t^-(v) \ge c]. 
        \end{align*}
        Then, since $Y_t^- \preceq X_t^-$ and $X_t^+ \preceq Y_t^+$,
        we obtain that
        \begin{align}
         \sum_{c\in [q]}  \Pr[X_t^+(v) \ge c] - \Pr[X_t^-(v) \ge c] 
		&\le \nonumber \sum_{c\in [q]}  \Pr[Y_t^+(v) \ge c] - \Pr[Y_t^-(v) \ge c]  \\
		&\le \nonumber \sum_{c\in [q]}  \left| \Pr[Y_t^+(v) \ge c] - \Pr[Y_t^-(v) \ge c] \right| \\
		&\le \nonumber q\|P_{\mathcal{B}}^t(+,\cdot) -P_{\mathcal{B}}^t(-,\cdot)  \|_{TV}\\
		&\le \label{eq:colortv} q\left(\|P_{\mathcal{B}}^t(+,\cdot) -\mu(\cdot)  \|_{TV} + 
		\|P_{\mathcal{B}}^t(-,\cdot) -\mu(\cdot)  \|_{TV}\right).		
	\end{align}
	Since $\mu$ is $\eta$-spectrally independent and 
	$b$-marginally bounded, 
	it follows from Theorem~\ref{thm:SI2KPF} and Remark~\ref{remark:kpf} that 
        $P_\mathcal{B}$ satisfies the relative entropy decay with rate 
	\begin{equation}
		\label{eq:mlsiblock}		
  		r \ge  \frac{b^{6\kappa}}{k \Delta^{4\kappa} \cdot (C (\eta+1)^5 )^{\kappa}},
	\end{equation}
 where $\kappa = 2+\lceil\frac{2\eta}{b}\rceil$.
	Let $b' :=  (C(\eta+1)^5/b^{6})^\kappa$, and let
	$$T:= k \Delta^{4\kappa} b' \log\left( \frac{\log (\mu_{min}^{-1})}{1/(4qn)} \right)
	=  O(\Delta^{4\kappa+1}b' \log (qn)).  
	$$ 
	By \eqref{eq:tmix} and \eqref{eq:mlsiblock}, 
	$T_{mix}(P_\mathcal{B}, 1/(8qn)) \le T$.
	Then for any $\sigma \in \Omega$,
	$$\|P_{\mathcal{B}}^T(\sigma,\cdot) -\mu(\cdot)  \|_{TV} \le \frac{1}{8qn},
	$$ 
	so we have $\Pr[X_T^+(v) \neq X_T^-(v)] \le 1/(4n)$.  
	By a union bound, $\Pr[X_T^+ \neq X_T^-] \le 1/4$, 
	and therefore
	\[T_{mix}(P_\phi) \le T =   O(\Delta^{4\kappa+1}b' \log (qn)),\]
establishing the desired bound for the mixing time.
\end{proof}
\subsection{Proof of the lower bound: Lemma~\ref{lemma:lbss}}
\label{sec:lb}

We provide next the proof of Lemma~\ref{lemma:lbss}.
Our proof extends the argument from \cite{HS07} for the Glauber dynamics and also uses ideas from \cite{BCSV22, BCPSV21}. The following fact will be used in our proof.
\begin{lemma}[Lemma 35, \cite{BCPSV21}]
\label{lemma:cmd}
    Let $\{X_t\}$ denote a discrete-time Markov chain with finite state space $\Omega$, reversible with respect to $\pi$ and with a positive semidefinite transition matrix. Let $B\subseteq \Omega$ denote an event.
    If $X_0$ is sampled proportional to $\pi$ on $B$, then $\Pr[X_t \in B]\ge \pi(B)$ for all $t\ge 0$, and for all $t\ge 1$,
    $$
        \Pr[X_t\in B]\ge \pi(B)+(1-\pi(B))^{-t+1}\left[\Pr(X_1\in B) - \pi(B) \right]^{t}.
    $$
\end{lemma}
\noindent We can now prove Lemma~\ref{lemma:lbss}. 
\begin{proof}[Proof of Lemma~\ref{lemma:lbss}]
Suppose $n$ is sufficiently large.
    Let  $R=\lceil \frac{\ln n}{8 \ln \Delta} \rceil$ and let $T=\alpha \ln n < R/3$ for some $\alpha>0$ we will specify later.
    We will show that for some (random) starting configuration $X_0\in \Omega$, 
    \begin{equation}
    \label{eq:tv}
            \| \ising(\cdot) - P^T_{EOE}(X_0, \cdot)\|_{TV} > 1/4,
    \end{equation}
    and hence by definition $T_{mix}(P_{EOE}) \ge T$.
    As $G$ has maximum degree $\Delta$, we can always find a subset $V_C\subseteq V$ of size at least $n^{1/4}$ whose pairwise graph distances are at most $2R$.
    Let $G_C := \cup_{u\in V_C} B(u, R)$.
    We consider a restriction of the even-odd scan dynamics on $G_C$.
    Let $\{X_t\}$ be an instance of the even-odd scan dynamics, 
    and let $\{Y_t\}$ be an even-odd scan dynamics that only updates spins for vertices in $G_C$, starting from the same configuration as $\{X_t\}$ which will be specified next.

    Let $N:=n^{1/4}$,
    and let $f:\Omega \rightarrow \mathbb{R}$ be the function given by $f(\sigma) = \frac{1}{N}\sum_{v\in V_C} \mathbbm{1}(\sigma(v) = +1)$.
    To show \eqref{eq:tv}, it suffices to find a distribution for $X_0 \in \Omega$ and a threshold $A\in \mathbb{R}$ such that 
    \begin{equation}
        \label{eq:threshold}
        \left| 
            \Pr\big[f(X_T) \ge A\big] - {\Pr}_{\sigma\sim \ising}\big[f(\sigma) \ge A\big] 
        \right| > 1/4.
    \end{equation}
        We define $X_0$ by setting
        the configuration on $V_C \cup (V \setminus G_C)$ to be the all $+1$ configuration and for each $v_C \in V_C$   
        sampling the configuration in $B(v_C, R) \setminus \{v_C\}$ conditional on the all $+1$ configuration on $V_C \cup (V \setminus G_C)$.
        Let $\pi$ denote the conditional distribution on $G_C$ with a fixed all $+1$ configuration on $V \setminus G_C$.
    Define $A:=\E_{\sigma\sim \pi}\left[f(\sigma) \right] + N^{-1/3}$.
    We will show next that 
    \begin{enumerate}
        \item $ {\Pr}_{\sigma\sim \ising}\big[f(\sigma) \ge A\big] \le 1/2$;
        \item  Under the identity coupling, $f(X_t) = f(Y_t)$ for $t\le T$. The identity coupling is the standard coupling that updates the same vertex in both chains at the same time and maximizes the probability that the spin value at the vertex agrees after the update;
        \item $\Pr\big[f(Y_T) \ge A\big] > \frac{3}{4},$
    \end{enumerate}
    and thus \eqref{eq:threshold} follows.
    
    We first give the upper bound for $ {\Pr}_{\sigma\sim \ising}\big[f(\sigma) \ge A\big]$.
    Since the ferromagnetic Ising model is monotone, and $f$ is a non-decreasing function, for any boundary condition $\tau$ on $\Omega_{V\setminus G_C}$,
    \[\E_{\sigma\sim \pi}\left[f(\sigma) \right] \ge \E_{\sigma\sim\ising^{\tau}}\left[f(\sigma) \right].\]
    For any $\tau \in \Omega_{V\setminus G_C}$, if $\sigma$ is generated from $\ising^\tau$, then $f(\sigma)$ is the average of $N$ independent indicator random variables. 
    By Hoeffding's inequality, 
    \[
     {\Pr}_{\sigma \sim \ising^\tau}\big[f(\sigma) \ge A \big] \le
    {\Pr}_{\sigma \sim \ising^\tau}\big[f(\sigma) \ge \E_{\sigma\sim\ising^\tau}\left[f(\sigma)\right] +N^{-1/3} \big] 
    \le \exp{\left(-\frac{2\cdot N^{4/3}}{N} \right)} < \frac{1}{2},   
    \]
    and thus
    $$ {\Pr}_{\sigma\sim \ising}\big[f(\sigma) \ge A\big]
    = \sum_{\tau \in \Omega_{V\setminus G_C}}  {\Pr}_{\sigma \sim \ising^\tau}\big[f(\sigma) \ge A \big] \cdot \mu(\tau)
        < \frac{1}{2}.
    $$

    To see that $f(X_t) = f(Y_t)$, we consider the speed of ``disagreement propagation''. 
    Note that $f(X_0) = f(Y_0)$ since $X_0=Y_0$. 
     The key observation is that under the identity coupling, in one step of the coupled even-odd scan dynamics, 
    the disagreement at any vertex $v$ can be propagated only to vertices at distance at most 3 from $v$.
    Since $R>3T$, we can guarantee that $X_{t}(v) = Y_t(v)$ for all $v\in V_C$ and all $t\le T$.

    Finally, we provide a bound for $\Pr\big[f(Y_T) \ge A\big]$.
    Fix $v \in V_C$. Let $\pi_v$ denote the Ising model distribution restricted to $B(v, R)$ 
    under the all $+1$ boundary condition outside of $B(v, R)$.
    Note that $\bigotimes_{v \in V_C} \pi_v = \pi$.
    Let $\{Y_t^{v}\}$ denote the Markov chain obtained by projecting $\{Y_t\}$ to $B(v, R)$.   
    Since the boundary of $B(v, R)$ is fixed, $\{Y_t^{v}\}$ is simply an even-odd scan dynamics on $B(v, R)$ under the all $+1$ boundary condition.
    It can be checked that $\{Y_t^{v}\}$ is reversible with respect to $\pi_v$ and that it has a positive semidefinite transition matrix.
    We define $\mathcal B_v$ to be the event (or subset of configurations) that $v$ is assigned spin $+1$. It can also be verified that $\ising$ is $b$-marginally bounded for some constant $b = b(\beta,\Delta)$, 
    so $b \le \pi_v(\mathcal B_v) \le 1-b$.
    Moreover, we have the following fact, which we prove later.
    \begin{claim}
        \label{claim:1step}
        There exists a constant $c:=c(\beta, \Delta) > 0$ such that $\Pr(Y^v_1\in \mathcal B_v) > \pi_v(\mathcal B_v) +c$.
    \end{claim}
    \noindent
    By Lemma~\ref{lemma:cmd} and Claim~\ref{claim:1step}, for all $t \ge 1$,
    \[        
    \Pr[Y_t^{v}\in \mathcal B_v]\ge \pi_v(\mathcal B_v) +b^{-t+1}\left[\Pr(Y_1^v\in \mathcal B_v) - \pi_v(\mathcal B_v) \right]^{t} \ge \pi_v(\mathcal B_v) + \frac{c^t}{b^{t-1}}.
    \]
   Using this and the definition of $f$, we have
    \begin{align*}
            \E[f(Y_T)] &= \frac{1}{N}\sum_{u \in V_C} \Pr[Y_T^{u} \in \mathcal B_u]
            \ge \frac{1}{N}\sum_{u\in V_C} \Big(\pi_u(\mathcal B_u) + \frac{c^T}{b^{T-1}}\Big)             
            = \E_{\sigma\sim \pi}\left[f(\sigma)\right]+ \frac{c^T}{b^{T-1}}.
    \end{align*}
    Set $T:= \min(\frac{R}{3}, \frac{\frac{1}{12} \ln n - \ln \frac{2}{b}}{\ln{\frac{b}{c}}})$, so that $\frac{c^T}{b^{T-1}} \ge 2N^{-1/3}$.
    Thus, 
    $
        \E[f(Y_T)] \ge  A + N^{-1/3}.
    $
   By Hoeffding's inequality, we obtain
    \[
        \Pr\big[f(Y_T) < A\big] \le   \Pr\left[f(Y_T) <  \E[f(Y_T)] - N^{-1/3} \right] 
        \le \exp\left[-\frac{2 N^{4/3}}{N} \right] < \frac{1}{4}.
    \]
    Therefore, the mixing time of $P_{EOE}$ is at least $T=\Omega(\log n)$.
\end{proof}
\noindent
It remains to prove Claim~\ref{claim:1step}.
\begin{proof}[Proof of Claim~\ref{claim:1step}]
    Let $P$ be the even-odd dynamics defined on $V' = B(v,R)$, and
	suppose $V'=V_E \cup V_O$ is a connected bipartite graph. 
    Suppose $v\in V_O$ without loss of generality.
    Recall that the transition matrix of $P$ is
    \[
    \prod_{i:v_i\in V_E} P_i \prod_{i:v_i\in V_O} P_i \prod_{i:v_i\in V_E} P_i.
    \]
    We use $Y_E, Y_{OE}$ and $Y_{EOE} = Y_1^v$ to denote the configuration of $Y^v_0$ after the updates $\prod_{i:v_i\in V_E} P_i$ on even vertices for the first time, after the updates $\prod_{i:v_i\in V_O} P_i$ on odd vertices and after update  $\prod_{i:v_i\in V_E} P_i$ respectively.
    Since the last set of updates on the even vertices do not affect the spin at $v$, we have
    \[
    \Pr(Y^v_1\in \mathcal B_v)
    = \E \big[\mathbbm{1} (Y_{EOE} \in  \mathcal B_v) \big] = \E \big[\mathbbm{1} (Y_{OE} \in  \mathcal B_v) \big]
    = \E \left[ \E \big[ \mathbbm{1} (Y_{OE} \in  \mathcal B_v) \mid Y_E \big] \right].
    \]
          Let $N(w)$ denote the set of vertices in $V'$ adjacent to $w$.
          For a configuration $\sigma\in \Omega$ and $w\in V$, we define $S(\sigma;w) := \sum_{x\in N(w)} \mathbbm{1} (\sigma_x = +1)$ and $g_w : \mathbb{Z} \rightarrow [0,1]$ given by 
      $g_w(y) := \ising(\sigma_w =+1 \mid S(\sigma; w) = y)$. 
    Let $\pi^+_v$ (resp. $\pi^-_v$) be distribution on $V'$ given by $\pi^+_v(\sigma) = \pi_v(\sigma \mid \sigma \in \mathcal B_v)$
     (resp.  $\pi^-_v(\sigma) = \pi_v(\sigma \mid \sigma \notin \mathcal B_v)$).
     Recall  that $Y^v_0$ is a configuration drawn from $\pi^+_v$
     and by noting that
     \[
        \pi^+_v \cdot \left( \prod_{i:v_i\in V_E} P_i\right) 
        =\pi^+_v, 
     \]
    so $Y_E$ can also be viewed as a configuration drawn from $\pi^+_v$.
    Hence, by the definition of the Gibbs update, we have
    \[
    \E \left[ \E \big[ \mathbbm{1} (Y_{OE} \in  \mathcal B_v) \mid Y_E \big] \right]
    = \E_{\tau \sim \pi^+_v} \left[ 
        g_v(S(\tau, v))
    \right].
    \]
    Similarly, 
    \[
    \pi_v(\mathcal B_v) = \E_{\sigma \sim \pi_v} \left[ 
        g_v(S(\sigma, v))
    \right].
    \]
    By Strassen's theorem, there exists a 
    coupling of $(\sigma, \tau)$ such that
    $\sigma \sim \pi_v, \tau\sim \pi^+_v$ and 
    $\sigma \le_q \tau$.
    Then $\sigma_{N(v)} \neq \tau_{N(v)}$ implies
    $S(\tau,v) \ge S(\sigma, v) + 1$.
    Therefore,
    \begin{align*}
             \Pr(Y^v_1\in \mathcal B_v) -  \pi_v(\mathcal B_v)
     &=  \E_{\tau \sim \pi^+_v} \left[ 
        g_v(S(\tau, v))
    \right] - \E_{\sigma \sim \pi_v} \left[ 
        g_v(S(\sigma, v))
    \right] \\
    &= \E_{(\sigma, \tau)\sim (\pi_v, \pi_v^+)}\left[
             g_v(S(\tau, v)) - g_v(S(\sigma, v))
        \right] \\
    &\ge \min_{i\le \deg(v)} (g_v(i,v) - g_v(i-1,v)) \cdot
\E_{(\sigma, \tau)\sim (\pi_v, \pi_v^+)}\left[
             S(\tau, v)) - S(\sigma, v)
        \right] \\
    &\ge  \min_{i\le \deg(v)} (g_v(i,v) - g_v(i-1,v)) \cdot
\E_{(\sigma, \tau)\sim (\pi_v, \pi_v^+)}\left[
\mathbbm{1}(\sigma_{N(v)} \neq \tau_{N(v)}) 
\right]. 
    \end{align*}
It can be checked that $\min_{i\le \deg(v)} (g_v(i,v) - g_v(i-1,v)) \ge c_2$, where $c_2:=c_2(\beta,\Delta)$,
Moreover,  for any $u\in N(v)$ we have
\[
    \E_{(\sigma, \tau)\sim (\pi_v, \pi_v^+)}\left[
        \mathbbm{1}(\sigma_{N(v)} \neq \tau_{N(v)}) 
    \right] \ge 
	\E_{(\sigma, \tau)\sim (\pi_v, \pi_v^+)}\left[
        \mathbbm{1}(\sigma_{u} \neq \tau_{u})
    \right].
\] 
Fix $u$ and let $\Lambda := V' \setminus \{u,v\}$.
Since $\sigma_u \le \tau_u$, 
$\sigma_{u} \neq \tau_{u}$ implies that $\sigma_u = -1$ and $\tau_u = +1$.
Thus we obtain
\begin{align*}
	\E_{(\sigma, \tau)\sim (\pi_v, \pi_v^+)}\left[
        \mathbbm{1}(\sigma_{u} \neq \tau_{u})
    \right] 
	&= \E_{(\sigma, \tau)\sim (\pi_v, \pi_v^+)} \left[
				\ising(\tau_u =+1 \mid \tau_\Lambda)-\ising(\sigma_u =+1 \mid \sigma_\Lambda)
		\right]\\
	&= \E_{(\sigma, \tau)\sim (\pi_v, \pi_v^+)} \left[
				g_u(S(\tau, u)) - g_u(S(\sigma, u))
		\right]\\
	&\ge b \cdot \E_{(\sigma, \tau)\sim (\pi_v^-, \pi_v^+)} \left[
		g_u(S(\tau, u)) - g_u(S(\sigma, u))
		\right],
\end{align*}
where the inequality is due to the $b$-bounded marginal condition of $\ising$ which
requires $\sigma_v=-1$ with probability at least $b$. 
Note that if $\sigma \sim \pi_v^-, \tau\sim \pi^+_v$ and 
$\sigma \le_q \tau$, then $S(\tau, u) \ge S(\sigma, u) + 1$.
Hence, 
\[
\E_{(\sigma, \tau)\sim (\pi_v^-, \pi_v^+)} \left[
		g_u(S(\tau, u)) - g_u(S(\sigma, u))
		\right]
		\ge  \min_{i\le \deg(u)} (g_u(i,u) - g_u(i-1,u)) > c_3,
\]
for some $c_3=c_3(\beta,\Delta)>0$.
Therefore, we established that 
\[
	\Pr(Y^v_1\in \mathcal B_v) -  \pi_v(\mathcal B_v) \ge c_2c_3b,
\]
and $c_2c_3b$ depends only on $\beta, \Delta$.
\end{proof}

\subsection{Applications of Theorem~\ref{thm:SIIsing:intro}}
\label{sec:sscor}
We discuss next some applications of Theorem~\ref{thm:SIIsing:intro}.
As a first application, we can establish \emph{optimal} mixing for 
the systematic scan dynamics on the ferromagnetic Ising model under the $\delta$-uniqueness condition, improving the best known results that hold under the Dobrushin-type conditions \cite{Simon93,DGJ08,Hayes06}.
This result was stated in Corollary~\ref{cor:tuIsing:intro} in the introduction and is proved next.
For this, we recall that under $\delta$-uniqueness condition, the Ising distribution $\ising$
satisfies spectral independence and the bounded marginals condition.
\begin{proposition} [\cite{CLV20,CLV21}]
	\label{prop:simb}
	The ferromagnetic Ising model with parameter $\beta$ such that $\bar{\beta}_u(\Delta)(1 - \delta)<\beta <\beta_u(\Delta)(1- \delta)$
	is $O(1/\delta)$-spectrally independent and $b$-marginally bounded with $b=O(1)$.
\end{proposition}

\begin{proof}[Proof of Corollary~\ref{cor:tuIsing:intro}]
	We fix $\delta\in(0,1)$ and first assume that $\Delta$ is a constant.
	By Proposition~\ref{prop:simb}, the ferromagnetic Ising model with parameter $\beta<(1-\delta)\beta_u(\Delta)$
	satisfies $\eta$-spectral independence and $b$-bounded marginals, 
	where $\eta = O(1/\delta)$ and $b$ is a constant.
	Since the ferromagnetic Ising model is a monotone system, it follows from Theorem~\ref{thm:SIIsing:intro} that $T_{mix} = O(\log n)$ for any ordering $\phi$.

	Now, when $\Delta \rightarrow \infty$ as $n \rightarrow \infty$, by Proposition~\ref{prop:Dobrushin}, the Dobrushin's influence matrix $A$ of ferromagnetic Ising model satisfies that $\|A\| \le 1 - \delta/2$.
    Under this assumption, it is known that $T_{mix} = O(\log n)$ for any ordering $\phi$; see \cite{Hayes06}. 
\end{proof}

We can similarly show mixing time bound for the systematic scan dynamics of the hardcore model on bipartite graphs under $\delta$-uniqueness condition. 
\begin{corollary}
	\label{cor:tuhardcore}
	Let $\delta \in (0,1)$ be a constant.
	Suppose $G$ is an $n$-vertex bipartite graph of maximum degree $\Delta \ge 3$.
	For the hardcore model on $G$ with fugacity $\lambda$ such that $0 < \lambda < (1-\delta)\lambda_u(\Delta)$, 
	where $\lambda_u(\Delta) = \frac{(\Delta-1)^{\Delta-1}}{(\Delta-2)^\Delta}$ is the tree uniqueness threshold on the $\Delta$-regular tree,  
	the systematic scan with respect to any ordering $\phi$ satisfies $$T_{mix}(P_\phi) = \Delta^{O(1/\delta)}\cdot O(\log n).$$
\end{corollary}
\begin{proof}[Proof of Corollary~\ref{cor:tuhardcore}]
	The hardcore model on a bipartite graph $(V_1\cup V_2, E)$ with fugacity $0 < \lambda < (1-\delta)\lambda_u(\Delta)$ is monotone, and \cite{CLV21,AJKPV22,CLY23} show that it satisfies $O(1/\delta)$-spectral independence and the $\Omega(\lambda)$-bounded marginals condition. 
	Theorem~\ref{thm:SIIsing:intro} then implies $\Delta^{O(1/\delta)}\cdot O(\log n)$ mixing of systematic scan for any ordering.
\end{proof}

We consider next 
the application of Theorem~\ref{thm:SIIsing:intro} to the special case where
the underlying graph is a cube of the $d$-dimensional lattice graph $\mathbb{Z}^d$. We show that strong spatial mixing implies optimal $O(\log n)$ mixing of any systematic scan dynamics.
Previously, under the same type of condition, \cite{BCSV18} gave
an $O(\log n (\log \log n)^2)$ mixing time bound for arbitrary orderings, 
and an $O(\log n)$ mixing time bound for a special class of scans that (deterministically) propagate disagreements slowly under the standard identity coupling. 
We first provide the definition of our SSM condition.

\begin{definition}
    \label{def:ssm}
We say a spin system $\mu$ on $\mathbb{Z}^d$ satisfies the  \emph{strong spatial mixing (SSM)} condition
if there exist constants $\alpha, \gamma, L > 0$ such that
for every $d$-dimensional rectangle $\Lambda \subset \mathbb{Z}^d$ 
of side length between $L$ and $2L$
and every subset $B\subset \Lambda$, 
with any pair $(\tau, \tau')$ of boundary configurations on $\partial\Lambda$  
that only differ at a vertex $u$, we have
\[
\| \mu_{B}^{\tau}(\cdot) - \mu_{B}^{\tau'}(\cdot)\|_{TV} \le \gamma \cdot \exp(-\alpha \cdot dist(u,B)),
\]
where $dist(\cdot, \cdot)$ denotes graph distance.
\end{definition}
\noindent
The definition above differs from other variants of SSM in the literature (e.g., \cite{DSVW04,BCSV18,MOS94})
in that $\Lambda$ has been restricted to ``regular enough'' rectangles.
In particular, our variant of SSM is easier to satisfy than those in \cite{DSVW04,MOS94} 
but more restricting than the one in \cite{BCSV18} (that only considers squares). 
Nevertheless, it follows from \cite{CP20,MOS94,Alx98,BDC12} that for the ferromagnetic Ising model, this form of SSM holds up to a critical threshold temperature $\beta < \beta_c(2) = \ln(1+\sqrt{2})$ on $\mathbb{Z}^2$.

Corollary~\ref{cor:monogrid:intro} from the introduction states that for $b$-marginally bounded monotone spin system on $d$-dimensional cubes $V\subseteq\mathbb{Z}^d$, SSM implies that the mixing time of any systematic scan $P_\phi$ is $O(\log n)$.
As mentioned there, this result in turn implies that any systematic scan dynamics for the ferromagnetic Ising model is mixing in $O(\log n)$ steps on boxes of $\mathbb{Z}^2$
when $\beta < \beta_c(2)$.
Another interesting consequence of Corollary~\ref{cor:monogrid:intro}
is that we obtain $O(\log n)$ mixing time for any systematic scan dynamics $P_\phi$ for the hardcore model on $\mathbb{Z}^2$ when $\lambda < 2.538$, which is the best known condition for ensuring SSM \cite{SSSY17,RSTVVY13}.

Our proof of Corollary~\ref{cor:monogrid:intro} relies on Lemma~\ref{lemma:SSMtoSI} that is restated below. 
Remarkably, Lemma~\ref{lemma:SSMtoSI} generalizes beyond monotone systems and may be of independent interests.
\ssmtosiintro*
\begin{proof}[Proof of Corollary~\ref{cor:monogrid:intro}]
Assume a monotone spin system satisfies SSM condition.
Then the spin system satisfies $\eta$-spectral independence, where $\eta = O(1)$ by Lemma~\ref{lemma:SSMtoSI}.
	By noting that $\Delta = 2^d$ 
	the corollary follows from Theorem~\ref{thm:SIIsing:intro}.
\end{proof}

Lastly, we give a proof of Lemma~\ref{lemma:SSMtoSI}. 
For this, we recall the notion of a $\kappa$-contractive coupling which is known to imply spectral independence.
We say a distribution $\mu$ is \emph{$\kappa$-contractive} with respect to a Markov chain $P$ if 
for all $X_0, Y_0 \in \Omega$,  
there exists a coupling of step of $P$ so that
\[
	\mathbb{E}[d(X_1, Y_1) \mid X_0, Y_0] \le \kappa d(X_0, Y_0),
\]
where $d(\cdot, \cdot)$ denotes the Hamming distance of two configurations.
The following lemma from \cite{BCCPSV22} shows that spectral independence follows from the existence of a contractive coupling with respect to a heat-bath block dynamics.
\begin{lemma} [\cite{BCCPSV22}]
	\label{lemma:coup2SI}
	If $\mu$ is $\kappa$-contractive with respect to a block dynamics,
	then $\mu$ is $(\frac{2DM}{1-\kappa})$-spectrally independent, where $M$ is the maximum block size and 
	$D$ is the maximum probability of a vertex being selected as part of a block in any step of the block dynamics.
\end{lemma}

With this lemma on hand, we can now prove Lemma~\ref{lemma:SSMtoSI}.
\begin{proof}[Proof of Lemma~\ref{lemma:SSMtoSI}] 
	Let $L$ be a sufficiently large constant so that the SSM condition is satisfied; we will choose $L$ later. 
	Let $V$ be a $d$-dimensional cube of $ \mathbb{Z}^d$.
	We define a heat-bath block dynamics $P_\mathcal{B}$ with respect to 
	a collection $\mathcal{B}$ of $d$-dimensional rectangles in $V$.
	Precisely, let $S_v:=\{w \in \mathbb{Z}^d:d_{\infty}(w,v)< L\}$, 
    and let $\mathcal{B}$ be the set of blocks $\{S_v \cap V\}_{v\in  V}$.
	Given a configuration $X_t$, the heat-bath block dynamics $P_\mathcal{B}$ obtains a configuration $X_{t+1}$ in 3 steps as follows:
	\begin{enumerate}
		\item Choose $v \in V$ uniformly at random. Let $S_v' := S_v \cap V$.
		\item Generate a configuration $\sigma \in \Omega_{S_v'}$ from $\mu_{S_v'}^{\tau}(\cdot)$, where $\tau\in \Omega_{V \setminus S_v'}$ is given by $\tau(u) = X_t(u)$;
		\item Let $X_{t+1}(u) = \sigma(u)$ if $u\in S_v'$ and $X_{t+1}(u) = X_t(u)$ otherwise.     
	\end{enumerate}   
	We will show that $\mu$ is $\kappa$-contractive with respect to $P_{\mathcal{B}}$ whenever SSM holds.
	Our argument builds upon \cite{DSVW04} but works for $P_\mathcal{B}$ under our weaker form of SSM condition, 
	in which the geometry is restricted to $d$-dimensional rectangles of large side lengths. 
	One can verify that if $\Lambda = S_v \cap V\in \mathcal{B}$, 
	then $\Lambda$ is a $d$-dimensional rectangle of side lengths between $L$ and $2L$. 
	The argument in \cite{DSVW04} requires a stronger form of SSM to deal with the set of blocks $\mathcal{B}' = \{ \Lambda = S_v \cap V: \Lambda \neq \emptyset, v \in \mathbb{Z}^d \}$ which contains arbitrarily thin rectangles,
    and this stronger form of SSM condition
    does not hold up to $\beta_c$ for the ferromagnetic Ising.

	Fix $(X_0, Y_0)$ such that there exists exactly one vertex $u\in V$ such that $X_0(u)\neq Y_0(u)$ and $X_0(v)=Y_0(v)$ for all $v\neq u$. 
	We select the same $v \in V$ in the first step of $P_{\mathcal{B}}$ in both chains; let $\Lambda = S_v'$.
	There are three cases with regard to the position of the disagreeing vertex $u$:
	$u$ is contained in $\Lambda$, $u$ is on the boundary of $\Lambda$, 
	or $u$ is far from $\Lambda$.
	Let $\partial \Lambda$ denote the external boundary of $\Lambda$.
	If $u \in \Lambda$ or  $u \notin (\Lambda \cup \partial\Lambda)$, since the boundary conditions are identical, 
	we generate the same configuration $\sigma \sim \mu_\Lambda^{\tau}$ to update $\Lambda$ in both chains such that 
	$X_1(\Lambda) = Y_1(\Lambda)$, where $\tau := X_0(\partial \Lambda) = Y_0(\partial \Lambda)$.
	Hence, 
	$\mathbb{E}[d(X_1, Y_1) \mid X_0, Y_0, u\in \Lambda] = 0$ and 
	$\mathbb{E}[d(X_1, Y_1) \mid X_0, Y_0, u \notin (\Lambda \cup \partial\Lambda)] = 1$. 
	
	It remains to define the coupling in the case when $u\in \partial\Lambda$, and we would need an upper bound for $\mathbb{E}[d(X_1, Y_1) \mid X_0, Y_0, u \in  \partial\Lambda]$.
	For this, we use the SSM condition.
	Let $B:=\{w\in \Lambda: d(w,u) \ge r \}$, where $r := \frac{1}{2}\left(\frac{L}{d}\right)^{1/2d}$,
	and let $\tau$ and $\tau'$ be the boundary conditions of $\Lambda$ in $X_0$ and $Y_0$ respectively.
	By assumption, $\tau$ and $\tau'$ are only different at $u$.
	We can view the coupling of the update on $\Lambda$ as consisting of three steps:
	\begin{enumerate}
		\item Generate two configurations  $\sigma_1, \sigma_2 \in \Omega_{B}$ from $\mu_{B}^{\tau}$ and $\mu_{B}^{\tau'}$ using the optimal coupling of the two distributions;
		\item Independently generate two configurations $\sigma_3, \sigma_4 \in \Omega_{\Lambda \setminus B}$ from $\mu_{\Lambda\setminus B}^{\tau \cup \sigma_1}$ and $\mu_{\Lambda\setminus B}^{\tau' \cup \sigma_2}$;
		\item Let $X_{1}(u) = \sigma_1(u)$ and $Y_{1}(u) = \sigma_2(u)$ if $u\in B$, 
		and $X_{1}(u) = \sigma_3(u)$ and $Y_{1}(u) = \sigma_4(u)$ if $u\in \Lambda \setminus B$.     
	\end{enumerate} 
	Clearly, $X_1(\Lambda) \sim \mu^{\tau}_{\Lambda}$ and $Y_1(\Lambda) \sim \mu^{\tau'}_{\Lambda}$,
	so the coupling is valid.
	By \eqref{eq:optcoupling}, 
	there exists a coupling $\Pp$ used for the first step such that
	$$
	\Pp[\sigma_1 \neq \sigma_2] = \|\mu^{\tau}_B -  \mu^{\tau'}_B\|_{TV}.
	$$
	Moreover, SSM implies that there exist constants $\gamma, \alpha >0 $ such that 
	$$
	\| \mu_{B}^{\tau} - \mu_{B}^{\tau'}\|_{TV} \le \gamma \cdot \exp(-\alpha \cdot dist(u,B)) \le \gamma \cdot e^{-\alpha r}.
	$$
	Also, $|\Lambda| \le (2L)^d$ and $|\Lambda \setminus B|\le (2r)^d$. 
	Put together, we have
	$$\mathbb{E}[d(X_1, Y_1) \mid X_0, Y_0, u\in \partial\Lambda] 
	\le 1+|\Lambda \setminus B| + |\Lambda|\cdot \Pp[\sigma_1 \neq \sigma_2]  
	\le 1 + (2r)^d + (2L)^d\cdot\gamma \cdot e^{-\alpha r}.$$
	Let $N:=|\mathcal{B}|$.
	Therefore, by noting that $\Pr[u\notin \Lambda] \ge L^d$ we obtain 
	\begin{equation}
		\begin{split}
			\mathbb{E}[d(X_1, Y_1) \mid X_0, Y_0] &= 
			\mathbb{E}[d(X_1, Y_1) \mid X_0, Y_0, u\in \partial\Lambda] \cdot \Pr[u\in \partial\Lambda]
			+ \mathbb{E}[d(X_1, Y_1) \mid X_0, Y_0, u\in \Lambda] \cdot \Pr[u\in \Lambda] \\
			&+ \mathbb{E}[d(X_1, Y_1) \mid X_0, Y_0, u\notin (\Lambda \cup \partial \Lambda)] \cdot \Pr[u\notin (\Lambda \cup \partial \Lambda)]\\
			& \le  1 + \Pr[u\in \partial\Lambda] \cdot [(2r)^d + (2L)^d\cdot\gamma \cdot e^{-\alpha r}] 
			- \Pr[u\in \Lambda]\\
			& \le 1 + \frac{2d\cdot(2L)^{d-1}}{N} \cdot [(2r)^d + (2L)^d\cdot\gamma \cdot e^{-\alpha r}] - \frac{L^d}{N}\\
			& = 1 + \frac{L^{d-1}}{N} \cdot \Bigg[ 2^d d\cdot \Big(\sqrt{\frac{L}{d}} + \frac{(2L)^d\cdot\gamma}{\exp(\alpha\cdot \sqrt[2d]{\frac{L}{d}}) }   \Big) - 2L \Bigg].
		\end{split}
	\end{equation}
	Recall that $N = O(n)$.
	By choosing $L=L(d,\alpha, \gamma)$ sufficiently large, we obtain 
	$$\mathbb{E}[d(X_1, Y_1) \mid X_0, Y_0]
	\le 1 - \frac{L^{d-1}}{N}
	= 1 - \Omega\left(\frac{1}{N}\right) = 1 - \Omega\left(\frac{1}{n}\right).$$

	In the case where blocks are of maximum size $(2L)^d$ and where each vertex is covered by at most $(2L)^d$ number of blocks at any step,
	$D = \Theta(n^{-1})$ and $M = O(1)$.
	Thus, Lemma~\ref{lemma:coup2SI} implies that $\mu$ is $\eta$-spectrally independent, where
	\[
		\eta = \frac{\Theta(n^{-1})}{1-\left(1-\Omega(n^{-1}) \right)} = O(1),
	\]
	as desired.
\end{proof}

\section{General block dynamics}
In this section, we give an upper bound for the mixing time of the block dynamics of a totally-connected spin system on general graphs. In particular, we prove Theorem~\ref{thm:blocks:intro} from the introduction.

We present next a more general form of entropy factorization.
In particular, KPF and UBF are special cases of it.
A Gibbs distribution $\mu$ is said to satisfy the \emph{general block factorization of entropy (GBF)} 
with constant~$C_{\mathrm{GBF}}$~if for all functions 
$f : \Omega \rightarrow \mathbb{R}_{\ge0}$, 
and 
for all probability distributions $\alpha$ over the set of all subsets of $V$,
\[
\alphamin \cdot \Ent_\mu (f) \le C_{\mathrm{GBF}} \sum_{U \subseteq V} \alpha(U) \E_{\tau \sim \mu_{V \setminus U}} \left[\Ent^\tau_{U} (f^\tau) \right],
\]
where $\alphamin = \min_{v\in V} \sum_{U: v\in U} \alpha(U)$.  
The notion of GBF is closely related to the general block dynamics~\cite{CP20, BCCPSV22,CMT15}. 
Indeed, the following proposition shows that a bound for $C_{\mathrm{GBF}}$ yields a bound for the modified log-Sobolev constant of general block dynamics.
\begin{proposition}[Lemma 2.8 in \cite{BCCPSV22}]
	\label{prop:GBF}
	If the Gibbs distribution $\mu$ of a spin system is totally-connected and satisfies GBF with constant $C_{\mathrm{GBF}}$,
	then the general block dynamics $P_{\mathcal{B},\alpha}$ w.r.t. $(\mathcal{B}, \alpha)$ satisfies relative entropy decay with rate at least $\frac{\alphamin}{C_{\mathrm{GBF}}}$ and satisfies a modified log-Sobolev inequality with constant $\rho(P_{\mathcal{B},\alpha}) \ge \frac{\alphamin}{C_{\mathrm{GBF}}}$.
\end{proposition}

The main theorem of this section is the following; Theorem~\ref{thm:blocks:intro} from the introduction follows as a corollary of this result.
\begin{theorem}
	\label{thm:GBF}
    Let $\eta>0, b>0, \Delta \ge 3$ and $\chi \ge 2$.
    Suppose $G=(V,E)$ is an $n$-vertex graph of maximum degree $\Delta$ and chromatic number $\chi$. Let $\mu$ be a Gibbs distribution of a totally-connected spin system on $G$.
        Let $\mathcal{B} := \{B_1, \dots, B_K\}$ be any collection of blocks such that $V=\cup_i B_i$, and let $\alpha$ be a distribution over $\mathcal{B}$.
    If $\mu$ is $\eta$-spectrally independent and $b$-marginally bounded, 
    then there exists a universal constant $C>1$ such that a general heat-bath block dynamics $P_{\mathcal{B}, \alpha}$ w.r.t. $(\mathcal{B}, \alpha)$ has
	modified log-Sobolev constant:
	\[\rho(P_{\mathcal{B}, \alpha}) = \Omega\left(\frac{\alphamin \cdot b^{6\kappa}}{\chi \cdot (C \Delta (\eta+1)^5 \log n )^{\kappa} \cdot}\right),\]
	where $\kappa = 2 +\lceil\frac{2\eta}{b}\rceil $, and 	
 $$T_{mix}(P_{\mathcal{B}, \alpha}) = O\left( \frac{\chi}{\alphamin} \cdot b^{-6\kappa}  \cdot \big(C (\eta+1)^5  \Delta\log n\big)^{\kappa}    \cdot \log n \right).$$
\end{theorem}

\noindent
Theorem~\ref{thm:GBF} follows from the bounds for $C_{\mathrm{KPF}}$ in Theorem~\ref{thm:SI2KPF} and the following lemma from \cite{BCCPSV22} that relates $k$-partite factorization with the general block factorization.

\begin{lemma}[Lemma 3.4, \cite{BCCPSV22}]
	\label{lemma:kpf2gbf}
	Suppose the Gibbs distribution $\mu$ of a spin system on a graph $G$ satisfies $k$-partite factorization of entropy with constant $C_{\mathrm{KPF}}$. 
	Then $\mu$ satisfies GBF with constant $k \cdot C_{\mathrm{KPF}}$.
\end{lemma}

\begin{proof} [Proof of Theorem~\ref{thm:GBF}]
The lower bounds for the entropy decay rate and MLSI constant follow from Theorem~\ref{thm:SI2KPF}, Lemma~\ref{lemma:kpf2gbf} and Proposition~\ref{prop:GBF}, and by \eqref{eq:tmix} we obtain the desired upper bound for mixing time.   
\end{proof}

We also obtain the following corollary for the ferromagnetic Ising and Potts model. 

\begin{corollary}\label{cor:GBF}
Let $\delta\in(0,1)$ and $\Delta \ge 3$.
	For the Ising model with $\beta \in[ (1-\delta)\bar{\beta}_u(\Delta), (1-\delta)\beta_u(\Delta)]$ 
	on any graph $G$ of maximum degree $\Delta$ and chromatic number $\chi$, 
	or the ferromagnetic $q$-state Potts model with $q\ge 2$ and 
	$0 < \beta \le \frac{2(1-\delta)}{\Delta}$ on the same graph, 
	$$T_{mix}(P_{\mathcal{B,\alpha}}) = O\left( \frac{\chi}{\alphamin} \right) \cdot O\left(\frac{\Delta}{\delta}\right)^{2+O(1/\delta)} \cdot \left( \log n\right)^{3+O(1/\delta)}.$$
\end{corollary}
\begin{proof}[Proof of Corollary~\ref{cor:GBF}]
    We have shown in the proof of Corollary~\ref{cor:SW} that, for the ferromagnetic $q$-state Potts model
    when $\beta$ is such that
    $0 < \beta \le \frac{2(1-\delta)}{\Delta}$, then $b=O(1)$ and $\eta = O(1/\delta)$. 
    For the Ising model, we achieve the same bound by Proposition~\ref{prop:simb}.
    Now 
    $\kappa = 2 +\lceil\frac{2\eta}{b}\rceil = 2+O(1/\delta)$,
    and the mixing time bound follows from Theorem~\ref{thm:GBF}.
\end{proof}

\bibliographystyle{alpha}
\bibliography{SSM}

\newcommand{\etalchar}[1]{$^{#1}$}
\begin{thebibliography}{COGG{\etalchar{+}}23}

\bibitem[AJK{\etalchar{+}}22]{AJKPV22}
Nima Anari, Vishesh Jain, Frederic Koehler, Huy~Tuan Pham, and Thuy-Duong
  Vuong.
\newblock Entropic independence: Optimal mixing of down-up random walks.
\newblock In {\em Proceedings of the 54th Annual ACM SIGACT Symposium on Theory
  of Computing}, STOC 2022, page 1418–1430, New York, NY, USA, 2022.
  Association for Computing Machinery.

\bibitem[Ale98]{Alx98}
Kenneth~S. Alexander.
\newblock On weak mixing in lattice models.
\newblock {\em Probab. Theory Relat. Fields}, 110(441-471), 1998.

\bibitem[ALOG20]{ALO20}
Nima Anari, Kuikui Liu, and Shayan Oveis~Gharan.
\newblock Spectral independence in high-dimensional expanders and applications
  to the hardcore model.
\newblock In {\em 2020 IEEE 61st Annual Symposium on Foundations of Computer
  Science (FOCS)}, pages 1319--1330, 2020.

\bibitem[AN05]{AN07}
Dimitris Achlioptas and Assaf Naor.
\newblock The two possible values of the chromatic number of a random graph.
\newblock {\em Annals of Mathematics}, 162(3):1335--1351, 2005.

\bibitem[BCC{\etalchar{+}}22]{BCCPSV22}
Antonio Blanca, Pietro Caputo, Zongchen Chen, Daniel Parisi, Daniel
  Stefankovic, and Eric Vigoda.
\newblock {On mixing of Markov chains: coupling, spectral independence, and
  entropy factorization}.
\newblock {\em Electronic Journal of Probability}, 27:1 -- 42, 2022.

\bibitem[BCP{\etalchar{+}}22]{BCPSV21}
Antonio Blanca, Pietro Caputo, Daniel Parisi, Alistair Sinclair, and Eric
  Vigoda.
\newblock {Entropy decay in the Swendsen–Wang dynamics on
  ${\mathbb{Z}^{d}}$}.
\newblock {\em The Annals of Applied Probability}, 32(2):1018 -- 1057, 2022.

\bibitem[BCSV19]{BCSV18}
Antonio Blanca, Pietro Caputo, Alistair Sinclair, and Eric Vigoda.
\newblock Spatial mixing and nonlocal markov chains.
\newblock {\em Random Structures \& Algorithms}, 55(3):584--614, 2019.

\bibitem[BCSV23]{BCSV22}
Antonio Blanca, Zongchen Chen, Daniel Stefankovic, and Eric Vigoda.
\newblock The {S}wendsen–{W}ang dynamics on trees.
\newblock {\em Random Structures \& Algorithms}, 2023.

\bibitem[BCT12]{BCT}
Christian Borgs, Jennifer~T. Chayes, and Prasad Tetali.
\newblock Tight bounds for mixing of the {S}wendsen-{W}ang algorithm at the
  {P}otts transition point.
\newblock {\em Probab. Theory Relat. Fields}, 152(3-4):509--557, 2012.

\bibitem[BCV20]{BCV18}
Antonio Blanca, Zongchen Chen, and Eric Vigoda.
\newblock Swendsen-wang dynamics for general graphs in the tree uniqueness
  region.
\newblock {\em Random Structures \& Algorithms}, 56(2):373--400, 2020.

\bibitem[BDC12]{BDC12}
Vincent Beffara and Hugo Duminil-Copin.
\newblock The self-dual point of the two-dimensional random-cluster model is
  critical for $q\ge 1$.
\newblock {\em Probab. Theory Relat. Fields}, 153(511-542), 2012.

\bibitem[BGP16]{BGP16}
Magnus Bordewich, Catherine Greenhill, and Viresh Patel.
\newblock Mixing of the glauber dynamics for the ferromagnetic potts model.
\newblock {\em Random Structures \& Algorithms.}, 48(1):21--52, January 2016.

\bibitem[BS15]{BSmf}
Antonio Blanca and Alistair Sinclair.
\newblock Dynamics for the mean-field random-cluster model.
\newblock In {\em Proceedings of {APPROX/RANDOM}}, 2015.

\bibitem[BT03]{BT03}
Sergey Bobkov and Prasad Tetali.
\newblock Modified log-sobolev inequalities, mixing and hypercontractivity.
\newblock In {\em Proceedings of the Thirty-Fifth Annual ACM Symposium on
  Theory of Computing}, STOC '03, page 287–296, New York, NY, USA, 2003.
  Association for Computing Machinery.

\bibitem[Ces01]{Cesi}
Filippo Cesi.
\newblock {Quasi-factorization of the entropy and logarithmic Sobolev
  inequalities for Gibbs random fields}.
\newblock {\em Probability Theory and Related Fields}, 120:569--584, 2001.

\bibitem[CFYZ22a]{CFYZ22b}
Xiaoyu Chen, Weiming Feng, Yitong Yin, and Xinyuan Zhang.
\newblock Optimal mixing for two-state anti-ferromagnetic spin systems.
\newblock In {\em 2022 IEEE 63rd Annual Symposium on Foundations of Computer
  Science (FOCS)}, pages 588--599. IEEE, 2022.

\bibitem[CFYZ22b]{CFYZ22a}
Xiaoyu Chen, Weiming Feng, Yitong Yin, and Xinyuan Zhang.
\newblock Rapid mixing of glauber dynamics via spectral independence for all
  degrees.
\newblock In {\em 2021 IEEE 62nd Annual Symposium on Foundations of Computer
  Science (FOCS)}, pages 137--148, 2022.

\bibitem[CGSV21]{CGSV21}
Zongchen Chen, Andreas Galanis, Daniel Stefankovic, and Eric Vigoda.
\newblock Rapid mixing for colorings via spectral independence.
\newblock In {\em Proceedings of the 2021 ACM-SIAM Symposium on Discrete
  Algorithms (SODA)}, pages 1548--1557, 2021.

\bibitem[CLMM23]{CLMM23}
Zongchen Chen, Kuikui Liu, Nitya Mani, and Ankur Moitra.
\newblock Strong spatial mixing for colorings on trees and its algorithmic
  applications, 2023.

\bibitem[CLV20]{CLV20}
Zongchen Chen, Kuikui Liu, and Eric Vigoda.
\newblock Rapid mixing of {G}lauber dynamics up to uniqueness via contraction.
\newblock In {\em 2020 IEEE 61st Annual Symposium on Foundations of Computer
  Science (FOCS)}, pages 1307--1318. IEEE Computer Society, 2020.

\bibitem[CLV21]{CLV21}
Zongchen Chen, Kuikui Liu, and Eric Vigoda.
\newblock Optimal mixing of glauber dynamics: Entropy factorization via
  high-dimensional expansion.
\newblock In {\em Proceedings of the 53rd Annual ACM SIGACT Symposium on Theory
  of Computing}, STOC 2021, page 1537–1550, New York, NY, USA, 2021.
  Association for Computing Machinery.

\bibitem[CLY23]{CLY23}
Xiaoyu Chen, Jingcheng Liu, and Yitong Yin.
\newblock Uniqueness and rapid mixing in the bipartite hardcore model, 2023.

\bibitem[CMT14]{CMT15}
Pietro Caputo, Georg Menz, and Prasad Tetali.
\newblock Approximate tensorization of entropy at high temperature, 2014.

\bibitem[COGG{\etalchar{+}}23]{PottsRGMetastabilityCMP}
Amin Coja-Oghlan, Andreas Galanis, Leslie~Ann Goldberg, Jean~Bernoulli
  Ravelomanana, Daniel {\v S}tefankovi{\v c}, and Eric Vigoda.
\newblock Metastability of the {P}otts ferromagnet on random regular graphs.
\newblock {\em Communications in Mathematical Physics}, 2023.

\bibitem[CP21]{CP20}
Pietro Caputo and Daniel Parisi.
\newblock Block factorization of the relative entropy via spatial mixing.
\newblock {\em Communications in Mathematical Physics}, 388(2):793--818, oct
  2021.

\bibitem[DGJ06a]{DGJ08}
Martin Dyer, Leslie~Ann Goldberg, and Mark Jerrum.
\newblock Dobrushin conditions and systematic scan.
\newblock In Josep D{\'i}az, Klaus Jansen, Jos{\'e} D.~P. Rolim, and Uri Zwick,
  editors, {\em Approximation, Randomization, and Combinatorial Optimization.
  Algorithms and Techniques}, pages 327--338, Berlin, Heidelberg, 2006.
  Springer Berlin Heidelberg.

\bibitem[DGJ06b]{DGJcolorings}
Martin Dyer, Leslie~Ann Goldberg, and Mark Jerrum.
\newblock {Systematic scan for sampling colorings}.
\newblock {\em The Annals of Applied Probability}, 16(1):185 -- 230, 2006.

\bibitem[DGJ09]{DGJnorms}
Martin Dyer, Leslie~Ann Goldberg, and Mark Jerrum.
\newblock {Matrix norms and rapid mixing for spin systems}.
\newblock {\em The Annals of Applied Probability}, 19(1):71 -- 107, 2009.

\bibitem[DSC96]{DS96}
P.~Diaconis and L.~Saloff-Coste.
\newblock {Logarithmic Sobolev inequalities for finite Markov chains}.
\newblock {\em The Annals of Applied Probability}, 6(3):695 -- 750, 1996.

\bibitem[DSVW04]{DSVW04}
Martin Dyer, Alistair Sinclair, Eric Vigoda, and Dror Weitz.
\newblock Mixing in time and space for lattice spin systems: A combinatorial
  view.
\newblock {\em Random Structures \& Algorithms}, 24(4):461--479, 2004.

\bibitem[FGYZ22]{FGYZ21}
Weiming Feng, Heng Guo, Yitong Yin, and Chihao Zhang.
\newblock Rapid mixing from spectral independence beyond the boolean domain.
\newblock {\em ACM Trans. Algorithms}, 18(3), oct 2022.

\bibitem[Fil91]{Fill91}
James~Allen Fill.
\newblock {Eigenvalue Bounds on Convergence to Stationarity for Nonreversible
  Markov Chains, with an Application to the Exclusion Process}.
\newblock {\em The Annals of Applied Probability}, 1(1):62 -- 87, 1991.

\bibitem[FK13]{FillKahn13}
James~Allen Fill and Jonas Kahn.
\newblock {Comparison inequalities and fastest-mixing Markov chains}.
\newblock {\em The Annals of Applied Probability}, 23(5):1778 -- 1816, 2013.

\bibitem[GJ97]{GoJe}
Vivek~K. Gore and Mark~R. Jerrum.
\newblock The {S}wendsen-{W}ang process does not always mix rapidly.
\newblock In {\em Proceedings of the Twenty-Ninth Annual ACM Symposium on
  Theory of Computing}, STOC '97, pages 674--681, New York, NY, USA, 1997.
  Association for Computing Machinery.

\bibitem[GJ17]{GuoJer}
Heng Guo and Mark Jerrum.
\newblock Random cluster dynamics for the {I}sing model is rapidly mixing.
\newblock In {\em Proceedings of the Twenty-Eighth Annual {ACM-SIAM} Symposium
  on Discrete Algorithms, {SODA} 2017}, pages 1818--1827, 2017.

\bibitem[GKZ18]{guoScans}
Heng Guo, Kaan Kara, and Ce~Zhang.
\newblock Layerwise systematic scan: Deep {B}oltzmann machines and beyond.
\newblock In {\em International Conference on Artificial Intelligence and
  Statistics}, pages 178--187. PMLR, 2018.

\bibitem[GL18]{GL1}
Reza Gheissari and Eyal Lubetzky.
\newblock Mixing times of critical two-dimensional {P}otts models.
\newblock {\em Comm. Pure Appl. Math}, 71(5):994--1046, 2018.

\bibitem[GLP19]{GLP}
Reza Gheissari, Eyal Lubetzky, and Yuval Peres.
\newblock Exponentially slow mixing in the mean{-}field {S}wendsen{--}{W}ang
  dynamics.
\newblock {\em Annales de l'Institut Henri Poincare (B)}, 2019.
\newblock Extended abstract appeared in Twenty-Ninth Annual ACM-SIAM Symposium
  on Discrete Algorithms (SODA 2018), pp. 1981--1988.

\bibitem[G{\v{S}}V15]{GSVmf}
Andreas Galanis, Daniel {\v{S}}tefankovi\v{c}, and Eric Vigoda.
\newblock Swendsen-{W}ang algorithm on the mean-field {P}otts model.
\newblock In {\em Proceedings of {APPROX/RANDOM}}, 2015.

\bibitem[Hay06]{Hayes06}
Thomas~P. Hayes.
\newblock A simple condition implying rapid mixing of single-site dynamics on
  spin systems.
\newblock In {\em The 47th Annual IEEE Symposium on Foundations of Computer
  Science (FOCS'06)}, pages 39--46, 2006.

\bibitem[HS07]{HS07}
Thomas~P. Hayes and Alistair Sinclair.
\newblock A general lower bound for mixing of single-site dynamics on graphs.
\newblock {\em The Annals of Applied Probability}, 17(3), Jun 2007.

\bibitem[Hub03]{huber2003bounding}
Mark Huber.
\newblock A bounding chain for {S}wendsen-{W}ang.
\newblock {\em Random Structures \& Algorithms}, 22(1):43--59, 2003.

\bibitem[JPV22]{JPV21}
Vishesh Jain, Huy~Tuan Pham, and Thuy-Duong Vuong.
\newblock Spectral independence, coupling, and the spectral gap of the glauber
  dynamics.
\newblock {\em Information Processing Letters}, 177:106268, 2022.

\bibitem[LPW06]{LevinPeresWilmer2006}
David~A. Levin, Yuval Peres, and Elizabeth~L. Wilmer.
\newblock {\em {Markov chains and mixing times}}.
\newblock American Mathematical Society, 2006.

\bibitem[Mar99]{Martinelli}
Fabio Martinelli.
\newblock {Lectures on Glauber dynamics for discrete spin models}.
\newblock {\em Lectures on probability theory and statistics (Saint-Flour,
  1997)}, 1717:93--191, 1999.

\bibitem[Mar19]{Marton19}
Katalin Marton.
\newblock Logarithmic sobolev inequalities in discrete product spaces.
\newblock {\em Combinatorics, Probability and Computing}, 28(6):919–935,
  2019.

\bibitem[MOS94]{MOS94}
F.~Martinelli, E.~Olivieri, and R.~H. Schonmann.
\newblock {For $2$-D lattice spin systems weak mixing implies strong mixing}.
\newblock {\em Communications in Mathematical Physics}, 165(1):33 -- 47, 1994.

\bibitem[MT06]{MT06}
Ravi Montenegro and Prasad Tetali.
\newblock Mathematical aspects of mixing times in markov chains.
\newblock {\em Foundations and Trends® in Theoretical Computer Science},
  1(3):237--354, 2006.

\bibitem[PW13]{PW}
Yuval Peres and Peter Winkler.
\newblock Can extra updates delay mixing?
\newblock {\em Communications in Mathematical Physics}, 323(3):1007--1016,
  2013.

\bibitem[RST{\etalchar{+}}13]{RSTVVY13}
Ricardo Restrepo, Jinwoo Shin, Prasad Tetali, Eric Vigoda, and Linji Yang.
\newblock Improved mixing condition on the grid for counting and sampling
  independent sets.
\newblock {\em Probab. Theory Relat. Fields}, 2013.

\bibitem[SIM93]{Simon93}
BARRY SIMON.
\newblock {\em The Statistical Mechanics of Lattice Gases, Volume I}.
\newblock Princeton University Press, 1993.

\bibitem[SJ89]{JS89}
Alistair Sinclair and Mark Jerrum.
\newblock Approximate counting, uniform generation and rapidly mixing markov
  chains.
\newblock {\em Information and Computation}, 82(1):93--133, 1989.

\bibitem[SSSY17]{SSSY17}
Alistair Sinclair, Piyush Srivastava, Daniel Stefankovic, and Yitong Yin.
\newblock Spatial mixing and the connective constant: Optimal bounds.
\newblock {\em Probab. Theory Relat. Fields}, 168:153--197, 2017.

\bibitem[SW87]{SW}
Robert~H. Swendsen and Jian-Sheng Wang.
\newblock Nonuniversal critical dynamics in {M}onte {C}arlo simulations.
\newblock {\em Phys. Rev. Lett.}, 58:86--88, 1987.

\bibitem[Ull14]{Ullrich14}
Mario Ullrich.
\newblock Rapid mixing of swendsen{\textendash}wang dynamics in two dimensions.
\newblock {\em Dissertationes Mathematicae}, 502:1--64, 2014.

\end{thebibliography}

\appendix

\section{Proof of the second part of Theorem~\ref{thm:SI2KPF}}
\label{sec:SI2KPF1}
In this appendix, we prove \eqref{eq:SI2KPF1} in Theorem~\ref{thm:SI2KPF}, which begins by extrapolating the proof of Lemma 3.3 in \cite{BCCPSV22} as Lemma~\ref{lemma:proofstep1}.
\begin{lemma} [\cite{BCCPSV22}]
	\label{lemma:proofstep1}
	Let $\theta \in (0,1]$ and $n \ge \frac{2}{\theta} (\frac{4\eta}{b^2} + 1)$.
	Let $G, \mu$, $V_1, \dots, V_k$ be as in the assumption of Theorem~\ref{thm:SI2KPF}.
	Let $S$ be a uniformly generated block of vertices of size $\lceil\theta n\rceil$, 
	and let $S_1, \dots, S_m$ be the connected components of ~$S$.
 Recall that 
 $C_S(v)$ denotes the unique connected component $S_i$ in $S$ that contains $v$ if such a component exists, otherwise set it to be the empty set.
	Suppose further that for $S_i \subseteq S$, 
	$\Gamma(S_i)$ takes the minimum value such that
 the following inequality holds 
 for an arbitrary pinning $\tau\in\Omega_{S_i^c}$ and any function $g:\Omega_{S_i}^\tau\rightarrow \mathbb{R}_{\ge0}$:
 \[
 		\Ent^{\tau}_{S_i}(g)\le \Gamma(S_i) \sum_{j=1}^k \E_{\xi\sim \mu_{S_i \setminus V_j}^\tau}\left[\Ent_{V_j\cap S_i}^{\xi \cup \tau} (g^{\xi}_{S_i\cap V_j})\right].
 \]
	Then, 
	\begin{equation}
		\label{eq:proofstep16}
		\Ent_\mu(f) \le \frac{C_{\mathrm{UBF}}}{\theta} \sum_{j=1}^k \E_{\tau\sim \mu}\left[\Ent_{V_j}^\tau (f) \right] \cdot G_j,
	\end{equation}
	where 
	\[G_j := \max_{W\subset V_j}\max_{v\in W} \E_S\left[\Gamma(C_S(v)) \mid V_j \cap S = W\right]\] 
	and 
	the expectation $\E_S$ is taken over the uniform generation of $S$.
\end{lemma}

\begin{proof} [Proof of \eqref{eq:SI2KPF1} in Theorem~\ref{thm:SI2KPF}]
In the same way that we prove \eqref{eq:SI2KPF2}, if $\Delta^2 > \frac{b^4n}{10e(4\eta+b^2)}$ then it follows from Lemma~\ref{lemma:CGSVgap} and Lemma~\ref{lemma:gap2KPF} that
\[
C_{\mathrm{KPF}}
\le \frac{3(\lceil 2\eta \rceil + 2)^{4\kappa}}{(2b^{4})^{\kappa}} \cdot   \big(\frac{10e(4\eta + b^2)}{b^2}\big)^{\kappa} \cdot 
\Delta^{2\kappa} \le \frac{(240e)^{4\kappa} \cdot (\lceil \eta \rceil + 1)^{5\kappa} \cdot \Delta^{2\kappa}}{b^{6\kappa}}.
\]
Now we assume $\Delta^2 \le \frac{b^2n}{10e(4\eta+b^2)}$.
Take $\theta = \frac{1}{5e\Delta^2} \ge \frac{2(4\eta + b^2)}{b^2 n} = \frac{2}{n} \cdot (\frac{4\eta}{b^2} + 1)$.
Theorem \ref{thm:SI2UBF} implies that 
\begin{equation}\label{eq:UBF3}
C_{\mathrm{UBF}}=\left(\frac{e}{\theta}\right)^{\lceil\frac{2\eta}{b}\rceil} = 
\Big(5e^2\Delta^2\Big)^{\lceil\frac{2\eta}{b}\rceil}. 
\end{equation}
Given Lemma~\ref{lemma:proofstep1}, 
to show \eqref{eq:SI2KPF1} it remains to provide an upper bound 
$G_j$ for each $j$.
There are two main steps for proving this bound.
First, we upper bound $G_j$ in terms of the size of connected components in $S$.
Under the assumptions of Theorem \ref{thm:SI2KPF}, 
$\mu$ is $\eta$-spectrally independent and $b$-marginally bounded.
These properties by definition preserve under any pinning.
In particular, 
for any $S_i \subseteq S$ and an arbitrary pinning $\tau\in\Omega_{V\setminus S_i}$,
$\mu_{S_i}^\tau$ is still $\eta$-spectrally independent and $b$-marginally bounded.
	Thus, 
 Lemma~\ref{lemma:CGSVgap} and Lemma~\ref{lemma:gap2KPF} imply that 
   \[
    \Gamma(S_i) \le \frac{3(\lceil 2\eta \rceil + 2)^{4\kappa}}{(2b^{4})^{\kappa}} \cdot |S_i|^\kappa,
    \] 
	and letting $\tilde{b} :=  \frac{3(\lceil 2\eta \rceil + 2)^{4\kappa}}{(2b^{4})^{\kappa}}$ we have
	\begin{equation}
		\label{eq:Gj}
		G_j	\le \tilde{b}\max_{W\subset V_j}\max_{v\in W} \E_S\left[ |C_S(v)|^{\kappa}\mid V_j \cap S = W\right].
	\end{equation}

	The second part of this proof analyzes the conditional expectation term above on the right-hand side of \eqref{eq:Gj}.
	We fix $v\in V$ (and hence fix $V_j$) and fix a feasible $W$ such that $v\in W \subseteq V_j$ and $|W| \le \lceil\theta n\rceil$.
	We say a set $T\subseteq V\setminus V_j$ is \emph{$W$-connected} if $T\cup W$ is connected in $G$, and we denote by
	$S'(v)$ the unique $W$-connected vertex-set in $S$ that is adjacent to $v$, if such set exists, otherwise an empty set.
         Clearly if $S'(v) = \emptyset$, then $C_S(v) = \{v\}$.
        Suppose $S'(v) \neq \emptyset$.
	Observe that $C_S(v) = S'(v) \cup (C_S(v) \cap W)$.
        Since $(C_S(v) \cap W)$ must be adjacent to 
        $S'(v)$ if $S'(v) \neq \emptyset$, 
        $ |C_S(v) \cap W| \le \Delta \cdot |S'(v)|$.
        Hence, $|C_S(v)|\le (\Delta + 1)|S'(v)|$.
        
	Furthermore, let $G_2 := (V, E\cup E_2)$, where $E_2$ is the set of pairs of vertices that are of distance at most $2$ in $G$.
 Note that the degree of any vertex in $G_2$ is at most $\Delta^2$. 
Let $C_{S_2}(v)$ be the unique connected component in $G_2[S]$ that contains $v$.
Notice that the set $S'(v)$ is always a subset of $C_{S_2}(v)$, regardless of the specific set $W$ we choose to fix.
        Hence, for any $x$, 
        \[
            {\Pr}_S\left[ |C_S(v)| \ge x \mid V_j \cap S = W  \right] \le 
            {\Pr}_S\left[ |S'(v)| \ge \frac{x}{\Delta+1} \mid V_j \cap S = W \right] \le
            {\Pr}_{S}\left[ |C_{S_2}(v)| \ge \frac{x}{\Delta+1} \right].
        \]
	Now we apply Lemma~\ref{lemma:g1} to estimate the last probability.
 For $\theta < \frac{1}{4e\Delta^2}$,
 \begin{align*}
 {\Pr}_{S}\left[ |C_{S_2}(v)| \ge \frac{x}{\Delta+1} \right] &\le 
 \frac{\lceil \theta n \rceil}{n}\sum_{k=0}^\infty (2e\Delta^2 \theta)^{\left\lfloor \frac{x}{\Delta+1}\right\rfloor + k - 1}\\
 &\le \frac{1}{2e\Delta^2} \left(\frac{1}{2}\right)^{\left\lfloor \frac{x}{\Delta+1}\right\rfloor} \cdot\sum_{k=0}^\infty \frac{1}{2} \\
 &\le \frac{1}{\Delta^2} \cdot 2^{-\frac{x}{\Delta+1}}.
\end{align*}
 Hence, we obtain 
\begin{align*}
    \E_S\left[ |C_S(v)|^{\kappa}\mid V_j \cap S = W\right] 
    &\le \sum_{x=1}^n x^{\kappa }  {\Pr}_S\left[ |C_S(v)| \ge x \mid V_j \cap S = W  \right]\\
    &\le \sum_{x=1}^n x^{\kappa }  {\Pr}_{S}\left[ |C_{S_2}(v)| \ge \frac{x}{\Delta+1} \right]\\
    &\le  \sum_{x=1}^n x^{\kappa} \cdot  \frac{1}{\Delta^2} \cdot 2^{-\frac{x}{\Delta+1}}\\
    &\le 4 \Delta^{2\kappa}.
\end{align*}
	Therefore, 
	$
		G_j \le 4\tilde{b} \Delta^{2\kappa}.
	$
	This bound on $G_j$ together with \eqref{eq:proofstep16} and \eqref{eq:UBF3} implies 
	\[C_{\mathrm{KPF}}\le  4\tilde{b} \Delta^{2\kappa}\cdot\Big(5e^2\Delta^2\Big)^{\kappa}
    =  \frac{12(\lceil 2\eta \rceil + 2)^{4\kappa}}{(2b^{4})^{\kappa}}\cdot\Big(5e^2\Delta^2\Big)^{\kappa}
    \cdot \Delta^{2\kappa}
 ,\] 
	concluding the proof.
\end{proof}

\section{Additional proofs}
\label{sec:addproofs}
\begin{proof} [Proof of Lemma~\ref{lemma:CGSVgap}]
Let $\eta_0,\eta_1,\dots,\eta_{n-2}$ be a sequence of reals.
We say a distribution $\mu$ is $(\eta_0, \eta_1, \dots, \eta_{n-2})$-spectrally independent if for every $0\le k \le n-2$, 
any $\Lambda \subseteq V$ of size $k$ and any pinning $\tau$ on $\Lambda$,
$
\lambda_1(\Psi^\tau_\mu) \le \eta_k.
$
    Theorem 6 and 8 from \cite{CGSV21}\footnote{Originally these theorems are given for coloring, but their proofs naturally extend to general totally-connected distributions.} state that 
    if $\mu$ is $(\eta_0, \eta_1, \dots, \eta_{n-2})$-spectrally independent,
    then the spectral gap of the Glauber dynamics is at least
\begin{equation}
\label{eq:si2gap}
   \frac{1}{n}\prod_{k=0}^{n-2} \left( 1 - \frac{\eta_k}{n-k-1} \right).
\end{equation}
We complete the proof by establishing suitable bounds for each $\eta_k$.
Per Definition~\ref{def:si}, we have 
$\eta_k \le \eta$ for all $k\in [0, n-2]$.
In addition, we will show that 
\begin{equation}
    \label{eq:crudesi}
        \eta_k \le (n - k - 1)\cdot \left( 1 - \frac{2b^4}{(n-k)^4} \right).
\end{equation}
As such, we will have that $\eta_k \le \min\{\eta, (n - k - 1)\cdot ( 1 - \frac{2b^4}{(n-k)^4} ) \}$,
and we would finish the proof of Lemma~\ref{lemma:CGSVgap} by plugging these bounds for $\eta_k$ into \eqref{eq:si2gap}:
\begin{align*}
   \frac{1}{n}\prod_{k=0}^{n-2} \left( 1 - \frac{\eta_k}{n-k-1} \right)
   &\ge    \frac{1}{n}\prod_{k=0}^{n-2} \left( 1 -  \min\big\{\frac{\eta}{n-k-1},  1 - \frac{2b^4}{(n-k)^4}  \big\} \right) 
   =  \frac{1}{n}\prod_{k=1}^{n-1} \left( 1 -  \min\big\{\frac{\eta}{k},  1 - \frac{2b^4}{(k+1)^4}  \big\} \right) \\
   & \ge  \frac{1}{n} \left(\prod_{k=1}^{\lceil2\eta\rceil+1}\frac{2b^4}{(k+1)^4}  \right)
   \left(\prod_{k=\lceil2\eta\rceil+2}^{n-1} \big(1 - \frac{\eta}{k}\big) \right) 
    \ge  \frac{1}{n} \left(\frac{2b^4}{(\lceil2\eta\rceil + 2)^4}  \right)^{\lceil2\eta\rceil + 1}
   \cdot \exp\left(-\sum_{k=\lceil2\eta\rceil + 2}^{n-1} \frac{2\eta}{k} \right)\\
   &  \ge  \frac{1}{n} \left(\frac{2b^4}{(\lceil2\eta\rceil + 2)^4}  \right)^{\lceil2\eta\rceil + 1}
   \cdot \exp\left(-2\eta \ln n\right) \ge \left(\frac{2b^4}{(\lceil2\eta\rceil +2)^4} \cdot \frac{1}{n} \right)^{1+\lceil2\eta\rceil}.
\end{align*}
Now we provide a proof for \eqref{eq:crudesi}.
Let $\tau$ be a pinning on $\Lambda$ with $|\Lambda| = k$, and let $U = V\setminus \Lambda$.
Theorem 8 of \cite{CGSV21} shows that 
\[
\lambda_1(\Psi^\tau_U) = (n - k -1)\cdot \lambda_2(\hat{P}_\tau),
\]
where $\hat{P}_\tau$ denotes the transition matrix of the \emph{local random walk} on $\mathcal P^\tau:=\{(u, s): u\notin \Lambda, s\in \Omega^\tau_u \}$ whose entries are given by
$\hat{P}_\tau((u,a), (v,b)):=\frac{\mathbbm{1}[u\neq v]}{n-k-1}\cdot\mu^{\tau \cup (u,a)}_{U\setminus \{u\}}(\sigma_v = b)$.
Let $\pi^\tau$ be a distribution on $\mathcal P^\tau$ given by $\pi^\tau(u,s) = \frac{1}{n - k}\cdot  \mu^{\tau}(\sigma_u = s)$. It is straightforward to verify that $\hat{P}_\tau$ is reversible with respect to $\pi^\tau$.
By the standard relationship between conductance and the eigenvalue of a reversible transition matrix in \cite{JS89}, we have
\[
1- \lambda_2(\hat{P}_\tau) \ge \frac{\Phi^2}{2},\] 
where
\[
\Phi:=\min_{S\subseteq \mathcal P^\tau: S\neq \emptyset, \pi^\tau(S)\le 1/2} \Phi_S,~\text{and}\quad
\Phi_S:=\frac{1}{\pi^\tau(S)} \sum_{x\in S} \sum_{y\notin S}\pi^\tau(x) \hat{P}_\tau(x, y).
\]
As $\mu$ is totally-connected, for any $S\subseteq \mathcal P^\tau$ such that $S\neq \emptyset$ and  $\pi^\tau(S)\le 1/2$, there exist $x\in S$ and $y \notin S$
such that $\hat{P}_\tau(x, y) > 0$.
Also, since $\mu$ is $b$-marginally bounded, we have $\pi^\tau(x) \ge b/(n-k)$ and $\hat{P}_\tau(x, y) \ge b/(n-k-1)$.
Hence, 
\[
\Phi \ge 2\min_{S\subseteq \mathcal P^\tau, \pi^\tau(S)\le 1/2} \min_{x\in S, y\notin S: \hat{P}_\tau(x, y) > 0} \pi^\tau(x) \hat{P}_\tau(x, y) \ge 2\cdot\frac{b}{n-k}\cdot \frac{b}{n-k-1} \ge \frac{2b^2}{(n-k)^2}.
\]
It follows that
\[
\frac{\lambda_1(\Psi^\tau_U)}{n-k-1} = 1 - \gap(\hat{P}_\tau) \le 1 - \frac{\Phi^2}{2} \le 1 - \frac{2b^4}{(n-k)^4},
\]
which establishes \eqref{eq:crudesi}.
\end{proof}
\begin{proof} [Proof of Lemma~\ref{lemma:gap2KPF}]
We say that $\mu$ satisfies the log-Sobolev inequality with constant $\rho_1$ if for all 
functions $f : \Omega \rightarrow \mathbb{R}_{\ge0}$, 
\begin{equation*}
\rho_1\Ent_\mu (f) \le  \frac{1}{n}\sum_{v \in V}
\E_{\tau \sim \mu_{V \setminus \{v\}}} 
\left[\var^\tau_{v} (\sqrt{f^\tau}) \right].
\end{equation*}
Recall that $C_{\mathrm{AT}}$ is the least constant such that for all functions $f : \Omega \rightarrow \mathbb{R}_{\ge0}$, 
\begin{equation*}
\Ent_\mu (f) \le C_{\mathrm{AT}} \sum_{v \in V}
\E_{\tau \sim \mu_{V \setminus \{v\}}} 
\left[\Ent^\tau_{v} (f^\tau) \right].
\end{equation*}
    Proposition 1.1 from \cite{CMT15} implies that 
    \begin{equation}
        \label{eq:at2mlsi}
        C_{AT} \le \frac{1}{\rho_1 n} 
    \end{equation}
    Moreover, \cite{DS96} shows that
    \begin{equation}
        \label{eq:lsi2gap}
        \frac{1-2\mu_{min}}{\log(1/\mu_{min} - 1)} \cdot \gamma
	\le \rho_1.
    \end{equation}
    If $\mu_{min} > 1/3$, then $\mu$ is a trivial distribution and $C_{AT} \le 1$.
    Thus, we may assume that $\mu_{min} \le 1/3$.
    Since $\mu$ is $b$-marginally bounded, we have
    \begin{equation}
    \label{eq:mlsi2gapb}
            \frac{1-2\mu_{min}}{\log(1/\mu_{min} - 1)} \ge \frac{1}{3\log(1/\mu_{min})}
            \ge \frac{1}{3n \log(b^{-1})}.
    \end{equation}
    It follows from \eqref{eq:at2mlsi}, \eqref{eq:lsi2gap} and \eqref{eq:mlsi2gapb} that 
    \begin{equation}
        \label{eq:at2gapb}
            C_{AT} \le \frac{3\log (b^{-1})}{\gamma}.
    \end{equation}
	 Observe that by Corollary~\ref{cor:monent}, if $v\in B$, then
	 $$\E_{\tau\sim \mu_{V\setminus\{v\}}} 
	\left[\Ent_{v}^{\tau} (f^{\tau}) \right] \le \E_{\tau\sim \mu_{ V\setminus B}} 
	\left[\Ent_{B}^{\tau} (f^\tau) \right].$$
 Hence, given $k$ disjoint independent sets $ U_1, \dots, U_k$ of $V$ such that
$\bigcup_{i=1}^k U_i = V$, we have 
		\begin{align}
			 \sum_{v \in V}
\E_{\tau \sim \mu_{V \setminus \{v\}}} 
\left[\Ent^\tau_{v} (f^\tau) \right]
            = \nonumber
			\sum_{j=1}^k \sum_{v\in U_j} \E_{\tau \sim \mu_{V \setminus \{v\}}} 
\left[\Ent^\tau_{v} (f^\tau) \right]
			\le n	\sum_{j=1}^k \E_{\tau \sim \mu_{V \setminus U_j}} 
\left[\Ent^\tau_{v} (f^\tau) \right].
		\end{align}
Equivalently, we obtain that 
\begin{equation}
    \label{eq:AT2KPF}
    C_{KPF} \le n \cdot C_{AT}.
\end{equation}
By \eqref{eq:at2gapb} and \eqref{eq:AT2KPF}, we establish the lemma.
\end{proof}

\end{document}